\documentclass[leqno,11pt]{amsart}
\usepackage{amssymb, amsmath,latexsym,amsfonts,amsbsy, amsthm,esint,comment}
\usepackage{color}
\usepackage{graphicx}
\usepackage{tikz}
\usetikzlibrary{arrows, calc, decorations.markings, decorations.pathreplacing, intersections, patterns,patterns.meta}
\usepackage{caption}
\usepackage{subcaption}
\usepackage[hidelinks]{hyperref}
\usepackage{enumitem}

\usepackage{pgfplots}
\pgfplotsset{compat=1.15}
\usepackage{mathrsfs}
\usepackage[foot]{amsaddr}

\makeatletter
\DeclareFontFamily{U}{tipa}{}
\DeclareFontShape{U}{tipa}{m}{n}{<->tipa10}{}
\newcommand{\arc@char}{{\usefont{U}{tipa}{m}{n}\symbol{62}}}%

\newcommand{\arc}[1]{\mathpalette\arc@arc{#1}}

\newcommand{\arc@arc}[2]{%
  \sbox0{$\m@th#1#2$}%
  \vbox{
    \hbox{\resizebox{\wd0}{\height}{\arc@char}}
    \nointerlineskip
    \box0
  }%
}                                                   
\makeatother

\let\pa=\partial
\let\al=\alpha

\let\g=\gamma
\let\d=\delta

\let\lam=\lambda

\let\s=\sigma

\let\f=\frac

\let\G= \Gamma

\let\S=\Sigma
\let\T=\Theta
\let\Om=\Omega

\let\e=\varepsilon
\let\pa=\partial

\let\ri=\rightarrow
\let\na=\nabla

\def\tcf{\widetilde{\mathcal{F}}}

\newcommand{\beq}{\begin{equation}}
\newcommand{\eeq}{\end{equation}}
\newcommand{\beqo}{\begin{equation*}}
\newcommand{\eeqo}{\end{equation*}}
\newcommand{\ben}{\begin{eqnarray}}
\newcommand{\een}{\end{eqnarray}}
\newcommand{\beno}{\begin{eqnarray*}}
\newcommand{\eeno}{\end{eqnarray*}}

\newtheorem{theorem}{Theorem}[section]
\newtheorem{definition}[theorem]{Definition}
\newtheorem{lemma}[theorem]{Lemma}
\newtheorem{proposition}[theorem]{Proposition}
\newtheorem{corol}[theorem]{Corollary}

\theoremstyle{remark}

\newtheorem{rmk}{Remark}[section]

\newcommand{\dist}{\mathrm{dist}}

\newcommand{\BR}{\mathbb{R}}

\newcommand{\ch}{\mathcal{H}^1}

\newcommand{\tu}{\tilde{u}}

\newcommand{\mres}{\mathbin{\vrule height 1.6ex depth 0pt width 0.13ex\vrule height 0.13ex depth 0pt width 1.1ex}}

\begin{document}

\title[Half-space Allen-Cahn solution]{Half-space minimizing solutions of a two-dimensional Allen-Cahn system}

\author{Zhiyuan Geng}
\address{Department of Mathematical Sciences, University of Arkansas, 850 West Dickson Street, Fayetteville, AR 72701}
\email{zgeng@uark.edu}
\thanks{The author was partially supported by AMS-Simons Travel Grant \#25014588}

\date{\today}

\begin{abstract}
This paper studies minimizing solutions to a two-dimensional Allen-Cahn system on the upper half-plane, subject to Dirichlet boundary conditions,
\begin{equation*}
    \Delta u-\nabla_u W(u)=0, \quad u: \mathbb{R}_+^2\to \mathbb{R}^2,\  u=u_0 \text{ on } \pa \mathbb{R}_+^2,
\end{equation*}
where $W: \mathbb{R}^2\to [0,\infty)$ is a multi-well potential. We give a complete classification of such half-space minimizing solutions in terms of their blow-down limits at infinity. In addition, we characterize the asymptotic behavior of solutions near the associated sharp interfaces.

\end{abstract}


\maketitle

\section{Introduction}\label{sec:intro}

\subsection{Mathematical formulation and main results.} In this paper, we investigate minimizing solutions to the system
\begin{equation}\label{eq:2D allen cahn}
    \Delta u-\nabla_u W (u)=0,\quad u:\overline{\mathbb{R}^2_+} \to \mathbb{R}^2,
\end{equation}
where $\mathbb{R}^2_+$ is the upper half-plane $\{(x,y)\in \mathbb{R}^2: y>0\}$. Here $W:\mathbb{R}^2\to \mathbb{R}_+\cup\{0\}$ is a nonnegative potential function that vanishes at finitely many points in $\BR^2$. \eqref{eq:2D allen cahn} is the Euler-Lagrange equation corresponding to the energy functional
\begin{equation}\label{energy functional}
J(u,\Omega):=\int_\Om \left( \f12|\na u|^2+W(u) \right)\,dz,\ \ \forall \text{ bounded Lipschitz domain } \Om\subset \BR^2_+. 
\end{equation}

We define \emph{minimizing solutions} in the sense of De Giorgi as follows.
\begin{definition}\label{def: minimizing sol}
A function $u\in H^1_{loc}(\overline{\BR^2_+};\BR^2)$  is a \emph{minimizing solution} of \eqref{eq:2D allen cahn} if, for every bounded Lipschitz domain $\Om\subset \BR_+^2$ and every $v\in H^1(\Om;\BR^2)$ satisfying $u|_{\Om}-v \in H_0^1(\Om;\BR^2)$, one has 
\begin{equation}
    J(u,\Om)\leq J(v,\Om).
\end{equation}
We also call $u$ a local minimizer of the functional \eqref{energy functional}.
\end{definition}

We give the following hypotheses on the potential function $W$. For the hypotheses ($\mathrm{H}_3$) and ($\mathrm{H}_3'$), we assume that one of them holds.
\begin{itemize}
\itemsep0.5em 
    \item[($\mathrm{H}_1$)] $W\in C^3(\mathbb{R}^2;[0,+\infty))$, $\{z\in\mathbb{R}^2:W(z)=0\}=\{a_1,a_2,...,a_N\}$, $\nabla_uW(u)\cdot u>0$ if $|u|>M$, and 
    \begin{equation*}
        c_1|\xi|^2\leq \xi^T \na^2 W(a_i)\xi\leq c_2|\xi|^2,\ \ i\in\{1,...,N\},
    \end{equation*}
    where $M,\ c_1,\ c_2$ are positive constants.
    \item[($\mathrm{H}_2$)] For $i\neq j$, $i,j\in \{1,...,N\}$, there exists a minimizing heteroclinic connection $U_{ij}\in H^1_{loc}(\BR;\BR^2)$ that minimizes the following 1D energy:
\beqo
J_1(U):=\int_{\BR}\left(\f12|U'|^2+W(U)\right)\,d\eta, \quad \lim\limits_{\eta\ri-\infty}U(\eta)=a_i,\ \lim\limits_{\eta\ri+\infty}U(\eta)=a_j.
\eeqo
Denote by $\sigma_{ij}$ the minimal energy $J_1(U_{ij})$. $\sigma_{ij}=\sigma_{ji}$ for any pair $i\neq j$.
    \item[($\mathrm{H}_3$)]   $N=3$ and the following triangle inequality holds:
\begin{equation}\label{tri ineq}
    \sigma_{ij}+\sigma_{jk}>\sigma_{ik}, \quad \forall \{i,j,k\}=\{1,2,3\}.
\end{equation}
We refer to \cite[Proposition 2.6 \& Proposition 2.12]{afs-book} for the sufficiency of \eqref{tri ineq} for the existence of all three $U_{ij}$.

\item[($\mathrm{H}_3'$)] 
\noindent $N>3$ and $\sigma_{ij}\equiv \sigma$ for any $i, j\in\{1,...,N\},\,i\neq j$. 
\end{itemize}

A Dirichlet boundary condition will be imposed on $\pa \BR^2_+$:
\begin{equation}
    \label{bdy cond}
    \qquad u=u_0\ \ \text{on }\pa \BR^2_+=\{(x,0):x\in\BR\},
\end{equation}
where $u_0\in C^3(\mathbb{R},\BR^2)$. We further assume that $u_0$ satisfies the following: 
\begin{equation*}
    \lim\limits_{x\ri\infty} u_0(x)=a_1,\quad  \lim\limits_{x\to-\infty}u_0(x)=a_2.
\end{equation*}
\begin{equation}\label{u0 finite energy}
    \int_{-\infty}^\infty \left(\f12|u_0'|^2+W(u_0)\right)\,dx<\infty.
\end{equation}

The energy bound \eqref{u0 finite energy} together with Hypothesis ($\mathrm{H}_1$) implies that
\begin{equation}\label{u_0 conv to a1 a2}
    \int_0^\infty |u_0(x)-a_1|^2\,dx+\int_{-\infty}^0 |u_0(x)-a_2|^2\,dx<\infty.
\end{equation}
We note that it is easy to find boundary data $u_0$ satisfying all these conditions. For example, one can choose any function converging to $a_1,\,a_2$ at an exponential rate $e^{-k|x|}$, with the same type of decay on its derivative.

We also assume that $u$ and $|\na u|$ are uniformly bounded. Specifically, there exists a constant $M>0$ such that  
\begin{equation}\label{uniform bound}
    \|u\|_{L^\infty(\BR^2_+;\BR^2)}+ \|\na u\|_{L^\infty(\BR^2_+;\BR^{2\times 2})}\leq M.
\end{equation}
In fact, uniform boundedness of $|\na u|$ follows from uniform boundedness of $|u|$, $C^2$ regularity of $u_0$, and standard elliptic regularity argument up to boundary. We state \eqref{uniform bound} in this form for convenience. 

We also use the following degenerate Riemannian metric on $\mathbb{R}^2$ induced by the Allen-Cahn energy,  
\begin{equation}\label{def:metric d}
    d(p,q):=\inf\left\{\sqrt{2} \int_0^1 |\g'(t)|\sqrt{W(\g(t))}\,dt: \g\in C^1([0,1],\mathbb{R}^2),\, \g(0)=p,\,\g(1)=q\right\}, \ \ \forall p,q\in\BR^2. 
\end{equation}
Intuitively, $d(p,q)$ quantifies the minimal energy needed to move from $p$ to $q$ in the phase plane. It is worth mentioning that $d(a_i,a_j)=\sigma_{ij}$.

The minimizing solution $u$ of \eqref{eq:2D allen cahn} subject to \eqref{bdy cond} is a diffuse analog of a half-space minimizing partition. For $v\in \mathrm{BV}_{loc}(\overline{\BR_+^2}; \{a_1,...,a_N\})$ and for every bounded Lipschitz domain $\Omega\subset \BR_+^2$, define 
\begin{equation}\label{functional:min par map}
\begin{split}
 J_*(v,\Om):= &  \sum\limits_{1\leq i<j\leq N}\sigma_{ij}\mathcal{H}^1(\pa^* v^{-1}(a_i)\cap\pa^*v^{-1}(a_j)\cap \Om)\\
 &+ \int_{\pa^* \Om\cap \pa \BR_+^2\cap\{x>0\}} d(\mathrm{tr}_\Om\,v,a_1)\,d\ch+ \int_{\pa^* \Om\cap \pa \BR_+^2\cap\{x<0\}} d(\mathrm{tr}_\Om\,v,a_2)\,d\ch.
 \end{split}
\end{equation}
Here $\pa^*$ denotes the reduced boundary, and $\mathrm{tr}_\Om\,v$ denotes the interior trace of $v$ on $\pa\Om$.

\begin{definition}\label{def: min par map}
$v\in  \mathrm{BV}_{loc}(\overline{\BR_+^2}; \{a_1,...,a_N\})$ is called a (Dirichlet) \emph{minimizing partition map} if, for every bounded Lipschitz domain $\Om\subset \BR_+^2$ and every $w\in \mathrm{BV}(\Om;\{a_1,...,a_N\})$ satisfying 
\begin{equation*}
    \mathrm{tr}_\Om v=\mathrm{tr}_\Om w,\qquad \ch\,\text{-a.e. on }\pa \Om\cap \BR_+^2,
\end{equation*}
one has
\begin{equation*}
    J_*(v,\Omega)\leq J_*(w,\Omega).
\end{equation*}    
\end{definition}

Define the blow-down map of $u$ at the scale $r\gg 1$ by
\begin{equation*}
    u_r(z):= u(rz),\quad z\in \BR_+^2.
\end{equation*}
Then $u_r$ is a local minimizer of the functional
\begin{equation*}
    J_r(u,\Omega):=\int_{\Omega} \left(\f{1}{2r}|\na u|^2+r W(u)\right)\,dz,
\end{equation*}
subject to the boundary condition 
\begin{equation*}
    u_r(x,0)= u_0(rx),\quad \forall x\in\BR.
\end{equation*}
Our first result classifies the minimizing solution $u$ in terms of its asymptotic behavior at infinity by studying the limit of $u_r$, which can be characterized by minimal cones associated with the functional \eqref{functional:min par map}.

\begin{theorem}\label{main thm: cls of blow down}
    Assume ($\mathrm{H}_1$), ($\mathrm{H}_2$) and either ($\mathrm{H}_3$) or ($\mathrm{H}_3'$). Let $u:\BR_+^2\to \BR^2$ be a minimizing solution of \eqref{eq:2D allen cahn} satisfying \eqref{bdy cond}--\eqref{uniform bound}. There exists a minimizing partition map $\tilde{u}\in \mathrm{BV}_{loc}(\overline{\BR_+^2};\{a_1,a_2,...,a_N\})$ in the sense of Definition \ref{def: min par map} such that $\na \tilde{u}\cdot z=0$ in the distributional sense, and 
    \begin{equation*}
        \lim\limits_{r\to \infty} \|u_r-\tilde{u}\|_{L^1(K)}=0,\ \forall K\Subset \BR_+^2,
    \end{equation*}
    where $u_r(z):=u(rz)$. In polar coordinates $z=(r\cos\theta,r\sin\theta)$, extend the definition of $\tilde{u}$ to $\pa \BR_+^2$ by setting
    \begin{equation}\label{def of hatu}
        \hat{u}(r,\theta)=\begin{cases}
        \tilde{u}(r,\theta), & 0<\theta<\pi,\\
        a_1, & \theta=0,\\
        a_2, & \theta=\pi.
        \end{cases}
    \end{equation}
   Then there exist $0\leq \al_1\leq \al_2\leq \pi$ with the following properties. If $\al_1<\al_2$, there is an index $i\in\{3,...,N\}$ such that
       \begin{equation*}
        \hat{u}(r,\theta)=\begin{cases}
        a_i, &  \al_1<\theta<\al_2,\\
        a_1, & 0 \leq\theta\leq \al_1,\\
        a_2, & \al_2\leq \theta\leq \pi.
        \end{cases}
    \end{equation*}
    Otherwise if $\al_1=\al_2$, then
    \begin{equation*}
        \hat{u}(r,\theta)=\begin{cases}
        a_1, & 0\leq\theta< \al_1,\\
        a_2, & \al_2< \theta\leq \pi.
        \end{cases}
    \end{equation*}

    Consequently, $\tilde{u}$ belongs to one of the following categories:
      \begin{enumerate}[label={\alph*)}]
        \item Constant: $\al_1=\al_2=0$, $\al_1=\al_2=\pi$ or $\al_1=0,\, \al_2=\pi$.
        \item Two-phase profile: $0<\al_1<\al_2=\pi$, $0=\al_1<\al_2<\pi$, or $0<\al_1=\al_2<\pi$.
        \item Triple junction: $0<\al_1<\al_2<\pi$.
    \end{enumerate}  
    Moreover, whenever $\al_1<\al_2$, it holds that $\al_2-\al_1\geq \f{2\pi}{3}$ when ($\mathrm{H}_3'$) holds and $\al_2-\al_1\geq \theta_3$ when ($\mathrm{H}_3$) holds, in which case $i=3$. Here $\theta_3$ is the angle determined by Young's law: 
        \begin{equation}\label{young's law}
            \f{\sin\theta_1}{\sigma_{23}}= \f{\sin\theta_2}{\sigma_{13}}= \f{\sin\theta_3}{\sigma_{12}}, \qquad\sum\limits_{i=1}^3 \theta_i=2\pi,\ \theta_i\in(0,\pi).
        \end{equation}

\end{theorem}

\begin{rmk}
    In cases a) and b), since the blow-down limit may not coincide with $a_1$ on $\{(x,0):x>0\}$ or with $a_2$ on $\{(x,0):x<0\}$, a boundary transition layer is allowed. $u$ may transition from a phase $a_i$ to $a_1/a_2$ within a thin layer near the boundary $\pa \BR_+^2$.
\end{rmk}

For any $0\leq \alpha\leq \pi$, let $l_\alpha$ denote the ray emanating from the origin in direction $\alpha$:
\begin{equation*}
    l_\alpha:=\{(r\cos\alpha,r\sin\alpha):r>0\}.
\end{equation*}

Take a constant $\gamma>0$. We define the diffuse interface as 
\begin{equation}\label{def:diff intef}
    \mathcal{D}_\g:=\{z\in\BR^2_+ : \min_i|u(z)-a_i|\geq \g\}.
\end{equation}
Generally speaking, $\mathcal{D}_\g$ contains points that do not belong to any phase $a_i$, which therefore represents the interfacial region that separates coexisting phases. From Theorem \ref{main thm: cls of blow down}, we know that $u$ resembles $\hat{u}$ at infinity; therefore, it is natural to expect $\mathcal{D}_{\g}$ to be close to the jump set of $\hat{u}$ written as $\mathcal{S}(\hat{u})=l_{\al_1}\cup l_{\al_2}$. Below we state our second main result, which confirms this property and refines the profile of $u$ near $\mathcal{S}(\hat{u})$.

\begin{theorem}\label{main thm: refined behavior near interface}
Assume ($\mathrm{H}_1$), ($\mathrm{H}_2$), either ($\mathrm{H}_3$) or ($\mathrm{H}_3'$), and the following extra hypothesis on $W$:
\begin{itemize}
    \item[($\mathrm{H}_4$)] For any $i\neq j$, the heteroclinic connection $U_{ij}$ in ($\mathrm{H}_2$) is unique up to translation. Moreover, for the operator $L_{ij}$ defined by
    \begin{equation*}
        L_{ij}v=-v''+\nabla^2 W(U_{ij}(x))v, \quad v\in H^2(\BR;\BR^2),
    \end{equation*}
    the kernel is one-dimensional, i.e. $\mathrm{ker}\,L_{ij}=\mathrm{span\,\{U_{ij}'\}}$.
\end{itemize}
Let $u:\BR_+^2\to \BR^2$ be a minimizing solution of \eqref{eq:2D allen cahn} satisfying \eqref{bdy cond}--\eqref{uniform bound}. Take $\hat{u}$ as the extended blow-down map of $u$ defined in Theorem \ref{main thm: cls of blow down} and suppose 
$$\mathcal{S}(\hat{u})=l_{\al_1}\cup l_{\al_2},\qquad\text{ 
for some }\  0< \al_1\leq \al_2< \pi.$$
For every sufficiently small $\g>0$, there exists a constant $C=C(u,\g,W)$ such that for sufficiently large $R$,
\begin{equation}\label{main thm 2:localization of diff interface}
    \left(\mathcal{D}_\g\cap \pa B_R \right) \subset\{(R\cos\theta,R\sin\theta): \theta\in(\al_1-CR^{-1}, \al_1+CR^{-1})\cup (\al_2-CR^{-1}, \al_2+CR^{-1})\}.
\end{equation}

Furthermore, suppose for some $\alpha\in (0,\pi)$, $l_\alpha\subset \mathcal{S}(\hat{u})$ separates the phases $a_i$ and $a_j$, with $a_i$ on the side $\theta<\al$ and $a_j$ on the side $\theta>\al$. Set 
$$\mathbf{e}_\al:= (\cos\al,\sin\al),\quad \mathbf{e}_\al^\perp:= (-\sin\al,\cos\al).$$ 
Then there exists a constant $h$ such that for every $L>0$ and every $\beta\in (0,1)$,
\begin{equation}\label{main thm 2:convergence to Uij}
    \lim\limits_{x\ri+\infty} \|u(x\mathbf{e}_\al+y \mathbf{e}_\al^\perp)-U_{ij}(y-h)\|_{C^{2,\beta}([-L,L];\BR^2)}=0.
\end{equation}
\end{theorem}

\begin{rmk}
    The nondegeneracy condition ($\mathrm{H}_4$) on the heteroclinic connection $U_{ij}$ was first explicitly imposed in \cite{schatzman2002asymmetric} to control the asymptotics of two-phase diffuse solutions along the sharp interface. It is not needed for Theorem \ref{main thm: cls of blow down}.
\end{rmk}

\begin{rmk}
    \eqref{main thm 2:localization of diff interface} and \eqref{main thm 2:convergence to Uij} imply that if there is no boundary sharp interface, $\mathcal{D}_\g\cap B_R$ is contained in an $O(1)$ neighborhood of $\bigcup\limits_{\al}l_\al$. However, Theorem \ref{main thm: refined behavior near interface} does not cover boundary interfaces, i.e. when $\al_i=0$ or $\pi$. The analogous convergence up to the wall to a finite translate of heteroclinic connection is not available. Indeed, suppose $\mathbf{e}_\al=(1,0)$ and $\lim\limits_{x\to +\infty}\|u(x,y)-U_{1j}(y-h)\|_{C^{2,\beta}_{loc}}=0$, then $\lim\limits_{x\to+\infty}|u_0(x)-a_1|>0$, which contradicts \eqref{u_0 conv to a1 a2}.
\end{rmk}

\subsection{Related results and outline of the paper}

For the scalar Allen-Cahn equation with a double-well potential, the classification of entire solutions is closely related to De Giorgi's famous conjecture \cite{Degiorgi1978}: a bounded global solution that is monotone in one direction is expected to be one-dimensional in dimensions $N\leq 8$. Numerous important contributions have been made to prove De Giorgi's conjecture and its variants, as well as to understand the relationship between the solutions of the Allen–Cahn equation and minimal surfaces; see for instance, \cite{ghoussoub1998conjecture,ambrosio2000entire,savin2009regularity,farina20111d,del2011giorgi, tonegawa2012stable,pacard2013stable,liu2017global,wang2017new,guaraco2018min,wang2019finite,wang2019second,florit2025stable} and the expository papers \cite{savin2010minimal,chan2018giorgi} for a detailed account. For the half-space scalar Allen-Cahn equation, solutions are naturally related to the half-space Bernstein theorem for graphical minimal hypersurfaces considered in \cite{edelen2022bernstein,du2023half}. Analogous rigidity results for Allen-Cahn were obtained in \cite{farina2010flattening,hamel2021half,du2024flat}.

In the vector-valued case, minimizing solutions are diffuse analogues of minimizing partitions. The scalar Modica's estimate, and hence the corresponding energy monotonicity formula, are generally unavailable for systems \cite{afs-book,bethuel2025asymptotics}. The convergence of diffuse minimizers to partitions is instead described by $\Gamma$-convergence \cite{Baldo,sternberg1988effect,fonseca1989gradient}, with Dirichlet boundary conditions treated in \cite{gazoulis}. For the two-dimensional case, recent developments have shed light into geometric and analytic description of fine structures of the Allen-Cahn solutions. Bethuel \cite{bethuel2025asymptotics} studied the critical points of \eqref{eq:2D allen cahn} and successfully generalized results on the regularity of interfaces from the scalar case to the two-dimensional vectorial case with the help of an asymptotic monotonicity formula. Fusco \cite{fusco2025connectivity} studied the connectivity of the diffuse interface and showed the existence of a connected optimal network that separates all phases. Alikakos and Fusco \cite{AF} investigated two examples with carefully designed Dirichlet boundary data and derived sharp energy lower and upper bounds and pointwise estimates.

Many earlier constructions of non-one-dimensional solutions relied on symmetry. Bronsard, Gui and Schatzman \cite{bronsard1996three} constructed a planar triple-junction solution in an equivariant class for a symmetric three-well potential. Gui and Schatzman \cite{gui2008symmetric} constructed a symmetric three-dimensional quadruple-junction solution, while Fusco \cite{fusco} obtained planar junctions under weaker rotational symmetry assumptions. Other constructions and related variational approaches can be found in \cite{sternberg1994local,flores2001higher,schatzman2002asymmetric, fusco2017layered,monteil2017metric}. More recently, Alikakos and the author \cite{alikakos2024triple} and Sandier and Sternberg \cite{ss2024} independently constructed entire minimizers with triple-junction structure at infinity without symmetry assumptions. In both works, a sequential blow-down limit is a minimizing three partition of the form $\sum_{i=1}^{3}a_i\chi_{D_i}$. The uniqueness of this blow-down and the asymptotic heteroclinic profiles along the sharp interfaces were subsequently established in \cite{geng2025uniqueness,geng2025rigidity}.

The half-space solution considered in the present paper (see Definition \ref{def: minimizing sol}) appears naturally as a blow-up map at the boundary phase-transition point of the minimizing solution to $\e \Delta u-\f{1}{\e}\nabla_u W(u)=0$ on a finite domain $\Omega$, where $u|_{\pa\Om}$ has a phase segregation structure. The flat boundary changes the sharp-interface problem in the sense that its energy includes the cost of mismatch between the BV trace of a limiting partition and the prescribed phases on the flat boundary. The limiting interface can therefore lie on the boundary as well as in the interior. Theorem \ref{main thm: cls of blow down} shows that every minimizer under our hypotheses has a unique homogeneous blow-down limit, with a constant, two-phase, or three-phase profile. In the three-phase case, the opening angle of the intermediate phase is bounded below by $\theta_3$, which is the angle determined by Young's law (this angle is $2\pi/3$ for equal surface tensions). It needs not equal that angle as in the case of interior junctions. For sharp interfaces strictly inside the half-plane, Theorem~\ref{main thm: refined behavior near interface} further establishes uniform distance localization of the diffuse interface and convergence to a single translate of the corresponding heteroclinic connection, under the nondegeneracy hypothesis ($\mathrm{H}_4$).

For the classification of blow-down limits, we develop a substantially simpler and more direct proof than the approaches based on slicing and diffuse-interface localization in \cite{alikakos2024triple,geng2025uniqueness} and comparison with wetting partitions in \cite{ss2024}. The argument works directly at the level of the limiting partition, combining compactness with geometric comparison to establish homogeneity and classify the possible blow-down profiles.

The paper is organized as follows. In Section \ref{sec:prelim},  a $\Gamma$-convergence argument gives both compactness and convergence of the energy measures, including the boundary mismatch term. Section \ref{sec:classification} then uses comparison for the limiting partition to obtain a limiting energy density and homogeneity. Comparing this limit with the one-dimensional lower bound on circular slices yields asymptotic energy equipartition and negligible radial energy. Classification of homogeneous minimizing partitions follows from a direct geometric variation and paired calibration arguments. 

In Section \ref{sec:uniqueness}, the divergence-free stress-energy tensor provides a simpler proof for the uniqueness of the blow-down limit, comparing \cite{geng2025uniqueness}. Its flux across a large semicircle converges to a vector given by a boundary integral and thus independent of the blow-down sequence. And the energy measure convergence and smallness of radial energy help identify the same vector as $\sum_{l}\sigma_{l}e_{\alpha_{l}}$, where $\al_l$ are jump angles for a blow-down limit. Thus the jump angles of every sequential blow-down are fixed, which immediately implies uniqueness. Finally, Section \ref{sec:sharp interface} obtains a sharp $O(1)$ bound for the energy excess along interior interfaces, which combined with the compactness and coercivity arguments near heteroclinic connections in \cite{schatzman2002asymmetric} yields the conclusion of Theorem \ref{main thm: refined behavior near interface}.

\section{Preliminaries}\label{sec:prelim}
\subsection{Notation}
We introduce the following notation for use throughout the paper. 
\begin{itemize}\itemsep5pt
    \item $z=(x,y)$ denotes a point in $\BR^2$. The origin $(0,0)$ is denoted by $O$.
    \item $B_r(z)$ is the disk centered at $z$ with radius $r$.  $B_r^+(z):=B_r(z)\cap \BR_+^2$. When $z$ is the origin, we simply write $B_r$ and $B_r^+$.
    \item Denote by $A^+_{r_1,r_2}(z_0)$ the part of annulus centered at $z_0$ in $\BR_+^2$: $A^+_{r_1,r_2}(z_0):=\{z\in \BR_+^2: r_1<|z-z_0|<r_2\}$. When $z_0$ is the origin, we simply write $A^+_{r_1,r_2}$.
    \item $|\cdot|$ denotes the ordinary modulus of a vector, for example $|u(z)-a_1|$. Denote by $d(\cdot,\cdot)$ the degenerate Riemannian metric on $\mathbb{R}^2$ defined in \eqref{def:metric d}.
    \item $\tilde{u}$ denotes a minimizing partition map in the sense of Definition \ref{def: min par map}.
    \item $\pa_r $ denotes the radial derivative, $\pa_\theta$ denotes the angular derivative, and $\pa_T:=\f{\pa_\theta }{r}$ denotes the counterclockwise unit tangential derivative.
    \item For $\alpha \in [0,\pi]$, let $l_\alpha$ denote the ray emanating from the origin in the direction $\alpha$, that is, $l_\alpha := \{ (r\cos\alpha, r\sin\alpha) : r > 0 \}.$
    \item Unless otherwise stated, $C$ denotes a constant that depends on the potential $W$, $M$ in the uniform bound \eqref{uniform bound} and sometimes also on the boundary data $u_0$. The value of $C$ may vary from line to line.
    \item For a Lipschitz domain $\Omega$ and a partition map $v\in \mathrm{BV}(\Om;\{a_1,...,a_N\})$, $\mathrm{tr}_\Om\, v$ denotes the interior trace of $v$ on $\pa^*\Om$, see, for instance \cite[Chapter 2]{giusti1984minimal}. When $\Omega=\BR_+^2$, we simply write $\mathrm{tr}\, v$ for the trace on the flat boundary $\pa \BR_+^2$.  
\end{itemize}

\subsection{Basic estimates for vector-valued Allen-Cahn minimizers}

First we recall the following 1D energy estimate.

\begin{lemma}[Lemma 2.3 in \cite{AF}]\label{lemma: 1D energy estimate}
There exist constants $\delta_W>0$ and $C=C(W)>0$ such that the following holds. Let $i\neq j \in \{1,...,N\}$, $\d <\d_W$ and $s_-<s_+$ be two real numbers. Let $v\in H^1((s_-,s_+);\BR^2)$ satisfy
\beqo
\max\{|v(s_-)-a_i|,|v(s_+)-a_j|\}\leq\delta.
\eeqo
Then
\beqo
J_{(s_-,s_+)}(v)=\int_{s_-}^{s_+}\left(\f12|v'(s)|^2+W(v(s))\right)\,ds\geq d(v(s_-),v(s_+))\geq \sigma_{ij}-C\delta^2.
\eeqo
\end{lemma}

Besides this one-dimensional energy estimate, we also need the following standard pointwise metric estimate, see for instance \cite[Proposition 2.1]{Baldo}. Let $\varphi_i(q):=d(q,a_i)$. For every bounded map $u\in H^1\cap L^\infty(\Om,\BR^2)$, one has
\begin{equation}\label{est: Dphi}
    |\na (\varphi_i\circ u)|\leq \sqrt{2W(u)}|\na u|\leq \frac{\e}{2}|\na u|^2 +\frac{1}{\e}W(u)\qquad \text{a.e. in }\ \Om.
\end{equation}
  
Next we present the vectorial maximum principle for $u$, which will be useful in our analysis.

\begin{lemma}[Variational maximum principle, \cite{AF2}]\label{lemma: maximum principle}
Assume ($\mathrm{H}_1$), ($\mathrm{H}_2$), and either ($\mathrm{H}_3$) or ($\mathrm{H}_3'$). There exists  $r_0=r_0(W)>0$ such that the following holds. Let $D\subset\mathbb{R}_+^2$ be a bounded, connected, Lipschitz domain, and let $u\in H^1(D;\BR^2)\cap L^\infty(D;\BR^2)$ minimize $J(\cdot,D)$ with respect to Dirichlet boundary data. If for some $r<r_0$ and $1\leq i\leq N$, 
\begin{equation*}
    | u(z)-a_i|\leq r \quad \ch\text{-a.e. on }\pa D, 
\end{equation*}
    then 
\begin{equation*}
        | u(z)-a_i|\leq r,\ \ \text{ a.e. in }\  D.
\end{equation*}
\end{lemma}

We also recall the following vectorial version of the Caffarelli--C\'{o}rdoba density estimate in \cite{AF3}.
\begin{lemma}[\cite{AF3}]\label{lem:cafarelli-cordoba}
Assume ($\mathrm{H}_1$), ($\mathrm{H}_2$), and either ($\mathrm{H}_3$) or ($\mathrm{H}_3'$). Let $u\in W^{1,2}_{loc}(\BR_+^2,\BR^2)$ be a minimizing solution of \eqref{eq:2D allen cahn} satisfying \eqref{uniform bound}. Suppose for some $r_0>0$, $\mu>0$, $0<\lambda<\min_{i\neq j}|a_i-a_j|$, $1\leq i\leq N$, $z\in\BR_+^2$,
\begin{equation*}
    \mathcal{L}^2(B_{r_0}(z)\cap\{|u-a_i|>\lambda\})\geq \mu,\quad B_{r_0}(z)\subset \BR_+^2,
\end{equation*}
then there exists $c=c(W,M,r_0,\lambda,\mu)>0$  such that 
\begin{equation*}
    \mathcal{L}^2(B_{r}(z)\cap\{|u-a_i|>\lambda\})\geq cr^2,\quad \forall\, r\in[r_0,\dist(z,\pa \BR_+^2)).
\end{equation*}
\end{lemma}

The following clearing-out lemma is a direct consequence of Lemma \ref{lem:cafarelli-cordoba}.

\begin{lemma}\label{lem: clear out}
    Assume ($\mathrm{H}_1$), ($\mathrm{H}_2$), and either ($\mathrm{H}_3$) or ($\mathrm{H}_3'$). Let $u$ be a minimizing solution of \eqref{eq:2D allen cahn} satisfying \eqref{uniform bound}. There exist constants $\delta_0,\,\e_0,\,r_0,\, K,\,k$, depending only on $W,\,M$, such that if for some $i$ and $r\geq r_0$, $z_0\in \BR_+^2$,
    \begin{equation}\label{clearout cond}
        B_{r}(z_0)\subset \BR_+^2,\qquad \fint_{B_{r}(z_0)} |u-a_i|\,dz\leq \e_0,
    \end{equation}
    then 
    \begin{equation}\label{est: clear out}
        \max\limits_{z\in B_{3  r/4}(z_0)}\,|u(z)-a_i|\leq \delta_0
    \end{equation}
    Moreover,
    \begin{equation}
        \label{est:exp decay}
        \max\limits_{z\in B_{r/2}(z_0)} \left( |u(z)-a_i|+|\na u|+|\na^2 u|\right)\leq Ke^{-kr}. 
    \end{equation}
\end{lemma}
\begin{proof}
Take $\delta_0\leq \f12\min_{i\neq j}|a_i-a_j|$. Suppose for some $z_1\in B_{3r/4}(z_0)$, $|u(z_1)-a_i|>\delta_0$. Then by \eqref{uniform bound},
\[
|u(z)-a_i|>\f{\d_0}{2},\qquad \text{on }B_{\f{\d_0}{4M}}(z_1).
\]
By Lemma \ref{lem:cafarelli-cordoba}, there exists $c=c(W,M)$ such that 
\[
\mathcal{L}^2(B_{r/4}(z_1)\cap\{|u-a_i|>\f{\d_0}{2}\})\geq cr^2,
\]
which contradicts \eqref{clearout cond} if we choose $\e_0$ suitably small. This proves \eqref{est: clear out}.

Then from \eqref{est: clear out} and the comparison principle for elliptic equations the exponential decay estimate \eqref{est:exp decay} follows (cf. \cite[Proposition 6.4]{afs-book}). 

\end{proof}

Next we derive a rough energy upper bound for $u$ on any disk $B_r(z)$.
\begin{lemma}\label{lem:rough ene upper bdd}
Assume ($\mathrm{H}_1$), ($\mathrm{H}_2$), and either  ($\mathrm{H}_3$) or ($\mathrm{H}_3'$). Let $u$ be a minimizing solution of \eqref{eq:2D allen cahn} satisfying \eqref{uniform bound}. Then there exists a constant $C=C(W,M)$ such that 
\begin{equation}\label{est: rough upper bdd}
    J(u,B_{r}(z)\cap \BR_+^2)\leq Cr, \quad \forall \,z\in \overline{\BR_+^2} \text{ and }r>0.
\end{equation}
\end{lemma}

\begin{proof}
   Fix $z_0\in \overline{R_+^2}$. Let $D$ denote $B_r(z_0)\cap \BR_+^2$. By \eqref{uniform bound}, $\f12|\na u|^2+W(u)\leq C(W,M)$. Therefore when $r\leq 2$, 
   \begin{equation*}
       J(u,D )\leq C\pi r^2 \leq Cr.
   \end{equation*}
   When $r>2$, define 
   \begin{equation*}
       \phi(z):=\min\{1,\dist\left(z,\pa D\right)\}.
   \end{equation*}
   Then $\phi(z)$ is 1-Lipschitz and vanishes on $\pa D$. Set $v\in H^1(D; \BR^2)$ as
   \begin{equation*}
       v(z):=(1-\phi(z)) u(z)+ \phi(z)\,a_1,
   \end{equation*}
   which is a competitor on $D$. Both $\nabla v(z)$ and $W(v(z))$ vanish when $\dist\left(z,\pa D\right)>1$. We have 
   \begin{align*}
       J(u,D)\leq \int_{D\cap\{\dist(z,\pa D)\leq 1\}} \left(\f12|\na v|^2+W(v)\right)\,dz\leq C\ch(\pa D)\leq Cr,
   \end{align*}
   which completes the proof.
   
\end{proof}

Moreover, we present the following compactness result. 

\begin{proposition}\label{prop: compactness of blow-down map}
    Assume ($\mathrm{H}_1$), ($\mathrm{H}_2$), and either ($\mathrm{H}_3$) or ($\mathrm{H}_3'$). Let $u:\BR_+^2 \to \BR^2$ be a minimizing solution of \eqref{eq:2D allen cahn} satisfying \eqref{bdy cond}--\eqref{uniform bound}. For any sequence $r_j\to\infty$, there exists a subsequence, still denoted by $\{r_j\}$, and a function $\tilde{u}\in \mathrm{BV}_{loc}(\BR_+^2; \{a_1,...,a_N\})$ such that 
    \begin{equation}\label{L1 comp}
        u_{r_{j}}(z) \to \tilde{u}(z),\ \text{ in } L_{loc}^1(\BR_+^2;\BR^2). 
    \end{equation}
    Here $u_{r_{j}}(z):=u(r_{j}z)$ and $\tilde{u}$ is a minimizing partition map  that locally minimizes $J_*(\cdot,\Omega)$ in the sense of Definition \ref{def: min par map}.

    Regarding the Radon measures induced by the energy density
    \begin{equation*}
        \mu_j:= \bigg(\frac{1}{2r_j}|\na u_{r_j}|^2+r_j W(u_{r_j})\bigg)\ \mathcal{L}^2\mres \BR_+^2,
    \end{equation*}
    it holds that 
    \begin{equation}\label{comp for measure}
        \mu_j\ \overset{\ast}{\rightharpoonup}\ \mu_{\tilde{u}}  \quad \text{ locally on }\overline{\BR_+^2}, 
    \end{equation}
    where $\mu_{\tilde{u}}$ is the energy measure associated with $J_*$, i.e.
    \begin{equation}\label{def:mu tu}
    \begin{split}
        &\mu_{\tilde{u}}:= \sum\limits_{1\leq i<j\leq N} \sigma_{ij}\ch\mres (\pa^* \tilde{u}^{-1}(a_i)\cap \pa^*\tilde{u}^{-1}(a_j)\cap\BR_+^2)\\
        &\qquad\qquad \qquad+ d(\mathrm{tr}\, \tilde{u},a_1)\ch\mres\{(x,0):x>0\}+  d(\mathrm{tr}\, \tilde{u},a_2)\ch\mres\{(x,0):x<0\}.
        \end{split}
    \end{equation}

    Consequently, for every bounded open set $A\subset \BR^2$ satisfying
    \begin{equation*}
        \mu_{\tilde{u}}(\pa A\cap \overline{\BR^2_+})=0,
    \end{equation*}
    we have 
    \begin{equation}\label{energy density conv}
        \lim\limits_{j\to\infty} J_{r_j}(u_{r_j},A\cap \BR_+^2)= \lim\limits_{j\to\infty}\f{1}{r_j}  J(u, (r_jA) \cap \BR_+^2)=\mu_{\tilde{u}}(A\cap \overline{\BR^2_+}).
    \end{equation}
\end{proposition}

\begin{proof}
We follow closely the compactness result of Sandier and Sternberg \cite[Proposition 3.1]{ss2024}, using the $\G$-convergence techniques in Baldo \cite{Baldo} and the treatment of Dirichlet boundary condition in Gazoulis \cite{gazoulis}.

Write 
\begin{equation*}
    u_j:=u_{r_j},\qquad g_j(x):=u_0(r_j\,x),\qquad g(x):=a_1\chi_{\{x>0\}}+a_2\chi_{\{x<0\}}.
\end{equation*}
Lemma \ref{lem:rough ene upper bdd} implies that 
\[
J_{r_j}(u_j,\overline{B_\rho^+})\leq C\rho, \qquad \forall \rho>0.
\]
Therefore, up to a possible subsequence, $\mu_j$ converges weakly to a Radon measure $\mu$ on $\overline{\BR_+^2}$. We extend $u_j$ below the $x$-axis (still denoted by $u_j$) by defining
\begin{equation*}
u_j(x,y)=g_j(x),\quad x\in \mathbb{R}, \ y\in (-1,0).    
\end{equation*}
For every $\delta\in (-1,0)$, 
\begin{equation*}
    J_{r_j}(u_j, \BR\times[\delta,0])=|\delta|\cdot \int_\BR \f12|u_0'|^2+W(u_0)\,dx=C_0|\delta|,
\end{equation*}
where $C_0=\int_\BR \f12|u_0'|^2+W(u_0)\,dx$. Using the classic compactness theorem for multiwell Allen-Cahn systems, see for example \cite[Proposition 4.1]{Baldo}, together with the uniform $L^\infty$ bound for $u_j$, we have 
\begin{equation*}
    u_j\to \tilde{u} \quad \text{in }L^1_{loc}(\{(x,y):y>-1\};\BR^2),
\end{equation*}
for some $\tilde{u}\in \mathrm{BV}_{loc}(\{(x,y):y>-1\};\{a_1,...,a_N\})$.  Since $\|g_j- g\|_{L^2(\BR;\BR^2)}\to\,0$, it holds that 
\[
\tilde{u}(x,y) =g(x), \quad a.e. \text{ on }\{(x,y):-1<y<0\}. 
\]
For the extended energy density measure
\[
\mu_j^*:=\bigg(\frac{1}{2r_j}|\na u_{r_j}|^2+r_j W(u_{r_j})\bigg)\ \mathcal{L}^2\mres \{(x,y):y>-1\},
\]
we have 
\[
\mu_j^*\ \overset{\ast}{\rightharpoonup}\ \mu+ C_0\ch\mres\{(0,y):-1<y<0\}. 
\]
By the lower bound part of $\G$-convergence, see \cite[(2.8)]{Baldo}, we have for any $\delta\in(-1,0)$
\begin{equation*}
    \mu+C_0\ch\mres\{(0,y):\delta< y<0\}\geq \sum_{1\leq i<j\leq N} \sigma_{ij} \ch\mres (\pa^* \tilde{u}^{-1}(a_i) \cap \pa^*\tilde{u}^{-1}(a_j)\cap\{y>\delta\} ). 
\end{equation*}
Taking $\delta \to 0-$, we get 
\begin{equation}\label{mu geq mutu}
    \mu\geq \mu_{\tilde{u}}  \quad\text{ on } \overline{\BR_+^2},
\end{equation}
where $\mu_{\tilde{u}}$ is defined by \eqref{def:mu tu}. Here the interface mismatch between the interior trace $\mathrm{tr}_{\pa\BR^2_+}\,\tilde{u}$ and the exterior phase $g$ produces the boundary term in \eqref{def:mu tu}.

For the upper bound, we can invoke the construction in \cite[proof of Theorem 1.1]{gazoulis} to obtain, for every $t>0$ satisfying
\begin{equation*}
    \mu_{\tilde{u}}(\pa B_t\cap \overline{\BR_+^2})=0,
\end{equation*}
a sequence of maps
\begin{equation*}
    w_j \in H^1(B_t^+;\BR^2)\cap L^\infty(B_t^+;\BR^2),\quad \sup_j \|w_j\|_{L^\infty(B_t^+)}<\infty,
\end{equation*}
such that 
\begin{align*}
    &\mathrm{tr}_{\{(x,0):-t<x<t\}} w_j=g_j,\qquad w_j\to \tilde{u} \quad \text{in }L^1(B_t^+;\BR^2),\\
    &\qquad \limsup\limits_{j\to\infty} J_{r_j}(w_j, B_t^+)\leq \mu_{\tilde{u}}(B_t\cap \overline{\BR^2_+}).
\end{align*}
Fix $s<t$ such that $\mu(\pa B_s\cap \overline{\BR_+^2})=0$. Since $u_j$ and $w_j$ both converge in $L^1$ (can be strengthened to $L^2$ due to uniform boundedness) to $\tilde{u}$ in the half-annulus $B_t^+\setminus B_s^+$, using the gluing estimate in \cite[Lemma 2.1]{gazoulis} one can construct a competitor in $B_t^+$, which agrees with $u_j$ on $\pa B_t\cap \BR^2_+$ and $w_j$ on $\pa B_s\cap\BR^2_+$, and obtain
\begin{equation*}
    J_{r_j}(u_j,B_s^+)\leq J_{r_j}(w_j,B_t^+)+o(1).
\end{equation*}
Letting $j\to\infty$ and then $t\downarrow s$ gives 
\[
\mu (B_s\cap\overline{\BR_+^2})\leq \mu_{\tilde{u}}(B_s\cap\overline{\BR_+^2}), \quad \forall s\text{ satisfying }\mu(\pa B_s\cap \overline{\BR_+^2})=0,
\]
which together with \eqref{mu geq mutu} gives \eqref{comp for measure}. The minimality of $\tilde{u}$ follows from a similar recovery sequence construction argument. For a competitor partition $w$ agreeing with $\tilde{u}$ outside a compact set, apply the same recovery construction to $w$. Choose the gluing annulus outside the compact set, where the recovery sequence $w_j$ and $u_j$ both converge to $\tilde{u}$. The minimality of $u_j$, the gluing estimate, and energy density measure convergence then give the minimizing property of $\tilde{u}$. This completes the proof.
    
\end{proof}

\section{Classification of blow-down limits}\label{sec:classification}

Throughout this section, we assume ($\mathrm{H}_1$), ($\mathrm{H}_2$), and either ($\mathrm{H}_3$) or ($\mathrm{H}_3'$). Let $u$ be a minimizing solution in the sense of Definition \ref{def: minimizing sol} satisfying \eqref{bdy cond}--\eqref{uniform bound}. Let $\tilde{u}\in \mathrm{BV}_{loc}(\BR_+^2;\{a_1,...,a_N\})$ be a sequential $L_{loc}^1$ blow-down limit of $u_{r_j}(z)$ obtained in Proposition \ref{prop: compactness of blow-down map}. In particular, $\tilde u$ is a minimizing partition map in the sense of Definition \ref{def: min par map}, and the energy measure convergence \eqref{comp for measure} holds.

Define the energy density of $\tu$ normalized by the radius $r$:
\begin{equation*}
    \T(r):= \frac{J_*(\tu, B_r\cap \overline{\BR_+^2})}{r}.
\end{equation*}

Since $\tu\in \mathrm{BV}_{loc}(\BR_+^2;\{a_1,...,a_N\})$, for almost every $r>0$, the angular slice  
\begin{equation*}
     h_r(\theta):=\tu(r\cos\theta,r\sin\theta)\in \mathrm{BV}((0,\pi);\{a_1,...,a_N\}). 
\end{equation*}
Let $\mathfrak{J}_r$ denote the discontinuity set of $h_r$ in $(0,\pi)$: 
\begin{equation*}
    \mathfrak{J}_r:= \{\theta\in (0,\pi): h_r(\theta-) \neq h_r(\theta+) \}.
\end{equation*}
Further define 
\begin{equation}\label{def:Qr}
    Q(r):= \sum_{\theta\in \mathfrak{J}_r} d(h_r(\theta-),h_r(\theta+))+ d(h_r(0+),a_1)+d(h_r(\pi-),a_2), 
\end{equation}
where the last two terms count the possible mismatches with the prescribed phase $g=a_1\chi_{\BR_+}+a_2\chi_{\BR_-}$ on $\pa \BR_+^2$. Finally denote by 
\begin{equation*}
    \g_{ij}:= \pa^*\tu^{-1}(a_i)\cap \pa^*\tu^{-1}(a_j)\cap \BR_+^2
\end{equation*}
the interior interface between $a_i$ and $a_j$ phases. 

The following lemma asserts the monotonicity for $\T(r)$.

\begin{lemma}\label{lemma:monotonicity}
    For the minimizing partition $\tu$, it holds that for almost every $r>0$, 
    \begin{equation}\label{eq: mon 1}
        \T(r)\leq Q(r).
    \end{equation}
    Moreover,
    \begin{equation}\label{eq: mon 2}
        (r\T(r))'\geq Q(r), \quad \text{ in the distributional sense}.
    \end{equation}
    Consequently, $\T(r)$ has a nondecreasing representative. If $\T(r)$ is constant on $(0,\infty)$, then $\tu$ is homogeneous, i.e. $\tu$ admits a representative depending only on $\theta$.
\end{lemma}

\begin{rmk}
    The statement of Lemma \ref{lemma:monotonicity} applies to any minimizing partition $v$ in the sense of Definition \ref{def: min par map}. Since our objective is to classify the blow-down limit $\tu$, we state the lemma and give the proof for $\tu$.  
\end{rmk}

\begin{proof}
    The lemma is a standard cone comparison monotonicity formula, see for example \cite[Theorem 8.2]{whitenotes} in the setting of minimizing flat chains. We present a short proof to account for the boundary mismatch in the half-plane setting.

    For almost every $r$, $h_r(\theta)$ is a $\mathrm{BV}$ function with finite jump set $\mathfrak{J}_r$ and satisfies
    \begin{equation}\label{good r}
        h_r(\theta)=\mathrm{tr}_{\pa B_r\cap\BR_+^2} \tilde{u},\quad h_r(0+)=\mathrm{tr}\, \tilde{u}(r),\quad h_r(\pi-)=\mathrm{tr}\, \tilde{u}(-r).
    \end{equation}
For such $r$, we consider the homogeneous extension of $h_r$ inside $B_r^+$
    \begin{equation*}
        w(t,\theta)=h_r(\theta),\quad \forall t\in(0,r),\ \theta\in(0,\pi).
    \end{equation*}
    Then $w$ is a competitor for $\tilde{u}$ in $B_r^+$, which by minimality of $\tu$ implies 
    \begin{equation*}
        J_*(\tu,B_r\cap \overline{\BR_+^2})\leq rQ(r),
    \end{equation*}
    which establishes \eqref{eq: mon 1}.

    We next set $\tau_{ij}$ as the unit tangent vector field along $\g_{ij}$. By coarea formula, for any test function $\eta\in C_C^\infty((0,\infty);[0,\infty))$,
    \begin{equation}\label{eq: mon 3}
    \begin{split}
      \int_{0}^{\infty} Q(r)\eta(r)\,dr&= \sum_{i<j}\sigma_{ij} \int_{\g_{ij}} \left|\tau_{ij}\cdot \f{z}{|z|}\right|\eta(|z|)\,d\ch+ \int_0^\infty  d(\mathrm{tr}\,\tu, a_1) \eta(x) \,dx \\
      &\qquad\qquad\qquad+  \int_{-\infty}^0 d(\mathrm{tr}\,\tu,a_2)\eta(|x|)\,dx\\
      &\leq \int_{\overline{\BR_+^2}} \eta(|z|)\,d\mu_{\tu}(z)=\int_0^\infty \eta(r) d(r\T(r)), 
      \end{split}
    \end{equation}
    which implies \eqref{eq: mon 2}. Combining \eqref{eq: mon 1} and \eqref{eq: mon 2} yields 
    \begin{equation*}
        \T(r)'=\f{(r\T(r))'}{r}-\f{\T(r)}{r}\geq \f{Q(r)-\T(r)}{r}\geq 0, \qquad \text{ distributionally.} \end{equation*}
    If $\T(r)\equiv c$ on $(0,\infty)$, then \eqref{eq: mon 1} and \eqref{eq: mon 2} imply that $Q(r)\equiv c$. In particular, \eqref{eq: mon 3} takes equality for every test function $\eta$, which implies 
    \begin{equation*}
       \left|\tau_{ij}\cdot \f{z}{|z|}\right|=1,\qquad \ch\text{-a.e. on }\g_{ij},\ \forall i<j. 
    \end{equation*}
    This further implies that $r\mapsto \tu(r,\theta)$ is constant for almost every $\theta\in(0,\pi)$, which proves the homogeneity of $\tu$.
    
    \end{proof}

\begin{proposition}\label{prop:homogenity of blowdown}
There exists $\al\geq \sigma_{12}$ such that 
\begin{align}
   \label{homog 1} &\qquad\lim\limits_{r\to\infty} \f{J(u,B_r^+)}{r}=\al,\\
   \label{homog 2}&\mu_{\tu}(B_r\cap \overline{\BR_+^2})=\al r\ \text{ for every }r\in(0,\infty). 
\end{align}
Consequently, every blow-down limit $\tu$ is homogeneous.
\end{proposition}

\begin{proof}
    Let $C_0$ be the constant in Lemma \ref{lem:rough ene upper bdd} such that $J(u,B_r^+)\leq C_0 r$. For any $t\in [\sigma_{12}, C_0]$, suppose there exists a radius $r$ such that $h_r(\theta)$ is a BV function satisfying \eqref{good r} and 
    \begin{equation*}
        Q(r)=t.
    \end{equation*}
    Then $t$ belongs to the discrete set
    \begin{equation*}
        \mathfrak{S}:=\{t\in [\sigma_{12},C_0]: \,\exists \{n_{ij}\}_{i<j}\subset \mathbb{N}_0=\{0,1,2,...\}\ \text{ s.t. }\sum_{i<j} \sigma_{ij}\,n_{ij}=t\},
    \end{equation*}
    since every term in the sum of $Q(r)$ is one of $\{\sigma_{ij},\ 1\leq i<j\leq N\}$. Here $Q(r)\geq \sigma_{12}$ follows from the triangle inequality and the boundary terms in \eqref{def:Qr}. 

    Assume by contradiction that the limit $\lim\limits_{r\to\infty} \f{J(u,B_r^+)}{r}$ does not exist. Choose two numbers $\al_1, \al_2$ such that
    \[
    \liminf_{r\to\infty}\f{J(u,B_r^+)}{r}<\al_1<\al_2<\limsup_{r\to\infty}\f{J(u,B_r^+)}{r},\qquad \al_2\not\in \mathfrak{S}.
    \]
    Then by continuity of $J(u,B_r^+)$, one can find two sequences $r_j\to\infty$, $R_j\to \infty$, such that 
    \begin{align*}
        &r_j<R_j,\quad \f{J(u,B_r^+)}{r}\leq \al_2 \ \text{ for }r\in[r_j,R_j],\\
        &\f{J(u,B_{r_j}^+)}{r_j}=\al_2,\quad \f{J(u,B_{R_j}^+)}{R_j}=\al_1.
    \end{align*}
    Therefore we have 
    \begin{equation*}
        \al_2r_j\leq \al_1R_j\quad \Rightarrow\quad\f{R_j}{r_j}\geq \f{\al_2}{\al_1}=:\kappa.
    \end{equation*}

    Using Proposition \ref{prop: compactness of blow-down map}, let $\tu$ be a sequentially blow-down limit of $u_{r_j}$ after passing to a subsequence, still denoted by $u_{r_j}$. For any $r\in (1,\kappa)$ such that $\mu_{\tu}(\pa B_r\cap \BR_+^2)=0$, it holds that 
    \begin{equation*}
        r\T(r)=\lim_{j\to\infty} \f{J(u,B_{r_jr}^+)}{r_j} \in [\al_2,\al_2r], 
    \end{equation*}
    which gives $\T(r)\leq \al_2$ for almost every $r\in (1,\kappa)$. On the other hand, monotonicity of $\T(r)$ implies 
    \begin{equation*}
        \T(s)\geq \T(r)\geq  \frac{\al_2}{r}, \quad\text{ for } 1<r<s<\kappa. 
    \end{equation*}
    Taking the continuity radii $r\downarrow 1$ yields $\T(s)\geq \al_2$ for almost every $s\in (1,\kappa)$. Hence,
    \[
    \T(r)\equiv \al_2, \quad r\in(1,\kappa).
    \]
    Lemma \ref{lemma:monotonicity} then implies $Q(r)\equiv \al_2$ almost everywhere on $(1,\kappa)$, contradicting  $\al_2\notin \mathfrak{S}$. Therefore $\f{J(u,B_r^+)}{r}$ must converge to some $\al>0$. 

    For any blow-down limit $\tu$ and almost every $r>0$, we have 
    \begin{equation*}
        \Theta(r)=\f{J_*(\tu,B_r\cap \overline{\BR_+^2})}{r}=\lim\limits_{j\to\infty} \f{J(u, B_{r_jr}^+)}{r_jr}=\alpha.
    \end{equation*}
    The homogeneity of $\tu$ follows immediately from Lemma \ref{lemma:monotonicity}. Its angular representative $\tu(\theta)=\tu(r,\theta)$ is a BV function on $(0,\pi)$ with finitely many jumps. \eqref{homog 2} is a direct consequence of $Q(r)\equiv \T(r)\equiv \al$ for a.e. $r$. Finally, $\al\geq \sigma_{12}$ thanks to the prescribed boundary data. 
    
\end{proof}

By Proposition \ref{prop:homogenity of blowdown}, we can write
\begin{equation*}
    \tu(r,\theta)=h(\theta) \in \mathrm{BV}((0,\pi);\{a_1,...,a_N\}), \quad r\in(0,\infty),\ \ \theta\in(0,\pi).
\end{equation*}
Set 
\[
\mathfrak{J}:=\{\theta\in(0,\pi):h(\theta-)\neq h(\theta+)\},
\]
as the set of jump angles. For the minimizing solution $u$, we define the quantities
    \begin{align*}
        &\mathcal{W}(r):= \f{1}{r}\int_{B_r^+} W(u)\,dz,\\
        &\mathcal{N}(r):=\f{1}{r}\int_{B_r^+} \f12 |\na u|^2\,dz.
    \end{align*}
The following lemma shows an asymptotic energy equipartition property for $u$ at large scales.

\begin{lemma}\label{lem:ene equip}
    Let $\al$ be the constant in \eqref{homog 1}. Then
    \begin{equation}\label{lim: W N}
    \lim\limits_{r\to\infty} \mathcal{W}(r)= \lim\limits_{r\to\infty} \mathcal{N}(r)=\f{\al}{2}. 
    \end{equation}
    Moreover, 
    \begin{equation}\label{lim: ene equip}
        \lim\limits_{r\to\infty} \f{1}{r} \int_{B_r^+}\left(\sqrt{W(u)}-\f{|\na u|}{\sqrt 2}\right)^2\,dz =0,
    \end{equation}
    and for every $0<\lambda_1<\lambda_2$,
    \begin{equation}\label{lim: small rad der}
        \lim\limits_{r\to\infty} \int_{A_{\lambda_1r,\lambda_2r}^+} \f{1}{|z|}|\pa_r u|^2\,dz=0.
    \end{equation}
\end{lemma}

\begin{proof}
    Fix $\lambda_2>\lambda_1>0$. Recall that $\tu$ is the $L^1_{loc}$ limit of $u_{r_j}$ for a sequence $r_j\to \infty$ and satisfies
    \begin{align*}
        &\qquad\qquad\qquad\tu(r,\theta)=h(\theta),\\
        &\lim\limits_{r\to\infty}\f{J(u,B_r^+)}{r}=\al=\sum_{\theta\in \mathfrak{J}}d(h(\theta-),h(\theta+)) + d(h(0+),a_1)+d(h(\pi-),a_2).
    \end{align*}
   It follows that 
    \begin{equation}\label{est:Ej}
        E_j:= \int_{A_{\lambda_1,\lambda_2}^+} \left(\f{1}{2r_j}|\na u_{r_j}|^2+r_jW(u_{r_j})\right)\,dz=\f{J(u,B_{\lambda_2r_j}^+)-J(u,B_{\lambda_1r_j}^+)}{r_j} \longrightarrow \al(\lambda_2-\lambda_1). 
    \end{equation}
    Set 
    \begin{equation*}
        F_j:= \int_{A_{\lambda_1,\lambda_2}^+} \sqrt{2W(u_{r_j})}\f{|\pa_\theta u_{r_j}|}{|z|}\,dz.
    \end{equation*}
    It is clear that $F_j\leq E_j$. Fubini's theorem implies, for almost every $r\in (\lambda_1,\lambda_2)$, $u_{r_j}(r,\cdot)$ (after extracting a possible subsequence) converges to $h(\theta)$ in $L^1(0,\pi)$, which further implies for such $r$, 
    \begin{equation*}
        \liminf_{j\to\infty} \int_0^\pi \sqrt{2W(u_{r_j}(r,\theta))}|\pa_\theta u_{r_j}(r,\theta)|\,d\theta\geq \al, 
    \end{equation*}
    since the left-hand side has to count the energy cost for all phase transitions at jump points of $h$ and possible boundary mismatches. This together with \eqref{est:Ej} gives
    \begin{equation*}
        0\leq \lim\limits_{j\to\infty} (E_j-F_j)\leq \al(\lambda_2-\lambda_1)-\al(\lambda_2-\lambda_1)=0.
    \end{equation*}
    Therefore,
   \begin{equation}\label{equip 1}
   \begin{split}
       &0=\lim_{j\to\infty} E_j-F_j\\
       &\qquad=\lim_{j\to\infty} \int_{A_{\lambda_1,\lambda_2}^+} \f{1}{2r_j}|\pa_r u_{r_j}|^2\,dz +\lim_{j\to\infty} \int_{A_{\lambda_1,\lambda_2}^+}\left|\f{1}{\sqrt{2r_j}}\f{|\pa_\theta u_{r_j}|}{|z|}- \sqrt{r_jW(u_{r_j})} \right|^2\,dz.
       \end{split}
   \end{equation}
   Furthermore,
   \begin{equation}\label{equip 2}
 \lim_{j\to\infty} \int_{A_{\lambda_1,\lambda_2}^+}\left|\f{|\na u_{r_j}|}{\sqrt{2r_j}}- \sqrt{r_jW(u_{r_j})} \right|^2\,dz
       \leq   \lim_{j\to\infty} E_j-F_j=0. 
   \end{equation}
   By Cauchy-Schwarz, we have 
   \begin{equation}\label{equip 3}
       \begin{split}
           &\left|\int_{A_{\lambda_1,\lam_2}^+} \f{|\na u_{r_j}|^2}{2r_j}\,dz-  \int_{A_{\lam_1,\lam_2}^+} r_jW(u_{r_j})\,dz \right|\\
           &\leq \left| \int_{A_{\lambda_1,\lam_2}^+} \left( \f{|\na u_{r_j}|}{\sqrt{2r_j}}- \sqrt{r_jW(u_{r_j})}  \right)\left(\f{|\na u_{r_j}|}{\sqrt{2r_j}} +  \sqrt{r_jW(u_{r_j})}\right) \,dz\right|\\
           &\leq \left(\int_{A_{\lambda_1,\lambda_2}^+}\left|\f{|\na u_{r_j}|}{\sqrt{2r_j}}- \sqrt{r_jW(u_{r_j})} \right|^2\,dz\right)^{\f12} \left( 2E_j \right)^{\f12} \\
           &\longrightarrow0,\quad \text{ as }j\to\infty.
       \end{split}
   \end{equation}
   Rescaling \eqref{equip 1}, \eqref{equip 2} and \eqref{equip 3} back to the original variables, we obtain
   \begin{align}
    \label{equip 4}   &\qquad\qquad\qquad\qquad\lim_{j\to\infty} \int_{A_{\lam_1r_j,\lam_2r_j}^+} \f{|\pa_r u|^2}{|z|}\,dz=0,\\
     \label{equip 5}  &\lim_{j\to\infty} \f{1}{r_j} \int_{A_{\lam_1r_j,\lam_2r_j}^+} \f{|\na u|^2}{2}\,dz =\lim_{j\to\infty} \f{1}{r_j} \int_{A_{\lam_1r_j,\lam_2r_j}^+} W(u)\,dz=\f{\lim\limits_{j\to\infty }E_j}{2}=\f{\al(\lam_2-\lam_1)}{2}, \\
\label{equip 6}       &\qquad\qquad\qquad \lim_{j\to\infty} \f{1}{r_j} \int_{A_{\lam_1r_j,\lam_2r_j}^+} \left(\f{|\na u|}{\sqrt 2}-\sqrt{W}\right)^2\,dz=0.
   \end{align}
\eqref{equip 4} gives \eqref{lim: small rad der} along the sequence $r_j\to\infty$. In \eqref{equip 5} and \eqref{equip 6}, we can take $\lambda_2=1$ and let $\lambda_1\downarrow 0$ to recover \eqref{lim: W N} and \eqref{lim: ene equip} along the same sequence $r_j\to\infty$. All conclusions still hold for $r\to \infty$, since for any $R_j\to\infty$, one can extract a subsequence satisfying the convergence results.     
   
\end{proof}

Next we give a classification of homogeneous minimizing partitions in $\BR_2^+$. Suppose $v\in \mathrm{BV}_{loc}$ is a minimizing partition in the sense of Definition \ref{def: min par map} and satisfies $v(\lambda z)=v(z)$ for almost every $z\in\BR_+^2$, $\lambda\in\BR_+$. Let $h_v(\theta)\in \mathrm{BV}((0,\pi);\{a_1,...,a_N\})$ be the angular representative of $v$ satisfying
\begin{equation*}
    v(r,\theta)=h_v(\theta). 
\end{equation*}

Under ($\mathrm{H}_3$), we define the angles 
\begin{equation*}
    \theta_i:= \arccos{(\f{\sigma_{jk}^2-\s_{ik}^2-\s_{ij}^2}{2\sigma_{ij}\s_{ik}})},\quad \{i,j,k\}=\{1,2,3\}.
\end{equation*}
Then the triangle inequality \eqref{tri ineq} implies $\theta_i\in (0,\pi)$ for every $i$, and these angles are the supplements of the angles for a triangle with three side lengths $\s_{12},\,\s_{23},\,\s_{13}$. In particular they satisfy Young's Law \eqref{young's law}

\begin{proposition}\label{prop: cla of blowdown}
    Assume ($\mathrm{H}_3$) or ($\mathrm{H}_3'$). For any homogeneous minimizing partition $v$ in Definition \ref{def: min par map}, its angular representative $h_v(\theta)$ must have one of the following forms.
    \begin{enumerate}
        \item For some $\al\in[0,\pi]$, 
        \begin{equation}
            \label{1 jump}
            h_v(\theta)=\begin{cases}
                a_1, & 0<\theta<\al,\\
                a_2, & \al<\theta<\pi.
            \end{cases}
        \end{equation}
        \item For some $i\in\{3,...,N\}$ and angles $0\leq \al_1<\al_2\leq \pi$,
        \begin{equation}
            \label{2 jump}
            h_v(\theta)=\begin{cases}
                a_1, & 0<\theta<\al_1,\\
                a_i,& \al_1<\theta<\al_2,\\
                a_2, & \al_2<\theta<\pi.
            \end{cases}
        \end{equation}
        Moreover, the angle of the third phase $a_i$ satisfies
        \begin{equation*}
            \al_2-\al_1 \geq \begin{cases}
                \theta_3, & \text{under }(\mathrm{H}_3), \\
                \f{2\pi}{3}, & \text{under }(\mathrm{H}_3').
            \end{cases}
        \end{equation*}
    \end{enumerate}
    In the above formula, empty angular intervals are omitted.

    Conversely, each angular function $h_v$ satisfying (1) or (2) defines a homogeneous minimizing partition. Thus, this gives a necessary and sufficient condition for homogeneous minimizing partitions.
\end{proposition}

\begin{proof}
    Take a minimizing partition $v(r,\theta)=h_v(\theta)$. We denote the jump angles of $h_v$ by 
    \begin{equation*}
        0\leq \al_1<\al_2<\cdots<\al_k\leq \pi.
    \end{equation*}
    Here $\al_1=0$ if $h_v(0+)\neq a_1$, which means there is a boundary mismatch between the interior trace of $v$ at $\{(x,0):x>0\}$ and $a_1$ phase. Similarly, the last jump $\al_k=\pi$ if $h_v(\pi-)\neq a_2$. 

    Let $i_0, i_1,..., i_{k-1},i_{k}$ be the indices of phases occurring counterclockwise:
    \begin{equation*}
        h_v(\theta)=\begin{cases}
         a_1,&\theta\in(0,\al_1),\\
         a_{i_j},&\theta\in (\al_{j}, \al_{j+1}),\ \ j=1,...,k-1,\\
         a_2, &\theta\in (\al_{k},\pi).
        \end{cases}
    \end{equation*}
   Set the endpoint indices as $i_0=1$, $i_k=2$. The ray $\{\theta=\al_j\}$ is an interface with weighted energy density $\sigma_{i_{j-1}i_{j}}$.

    Consider two adjacent jump angles $\al_j<\al_{j+1}$. For convenience we write 
    \begin{equation*}
    a_{i_{j-1}}=a,\quad a_{i_j}=b,\quad a_{i_{j+1}}=c.    
    \end{equation*}
    By definition, $b$ is the phase between $\al_j$ and $\al_{j+1}$, $a$ is the previous phase on the right and $c$ is the next phase on the left, see Figure \ref{fig: two adjacent jump}.

\begin{figure}[htb!]
    \centering
    \begin{tikzpicture}[scale=1.8]

\def\R{2}

\coordinate (O) at (0,0); 
\coordinate (S) at (-0.06,0.25); 

\draw[dashed] (-\R,0) arc[start angle=180,end angle=130,radius=\R] ;
\draw[dashed] ({\R*cos(80)},{\R*sin(80)}) arc[start angle=80,end angle=0,radius=\R] ;
\draw[thick] ({\R*cos(130)},{\R*sin(130)}) arc[start angle=130,end angle=80,radius=\R] ;

\draw[thick] (O) -- (0,0); 
\draw[thick] (O) -- ({\R*cos(0)}, {\R*sin(0)});
\draw[thick] (O) -- ({\R*cos(80)}, {\R*sin(80)});
\draw[thick] (O) -- ({\R*cos(130)}, {\R*  sin(130)});
\draw[thick] (O) -- ({\R*cos(180)}, {\R*sin(180)});
\draw[->] (O)--({0.7*cos(105)},{0.7*sin(105)}) node [right] {$\mathbf{\xi}$};

\draw[red,dashed] (S) -- (0,0);
\draw[red,dashed] (S) -- ({\R*cos(80)}, {\R*sin(80)});
\draw[red,dashed] (S) -- ({\R*cos(130)}, {\R*sin(130)});

\node at (1,0) [below]{$a_1$};

\node at (-1,0) [below]{$a_2$};

\node at ({1.2*cos(60)}, {1.2*sin(60)}) {$a_{i_{j-1}}\,(a)$};
\node at ({1.2*cos(146)}, {1.2*sin(146)}) {$a_{i_{j+1}}\,(c)$};
\node at (-0.4,1.2) {$a_{i_j}\,(b)$};
\node at (0,0) [below] {$(0,0)$};
\node at ({2.1*cos(80)}, {2.1*sin(80)}) {$\al_j$};
\node at ({2.2*cos(130)}, {2.2*sin(130)}) {$\al_{j+1}$};

\end{tikzpicture}
\caption{Each phase angle must be at least the corresponding Young angle. }\label{fig: two adjacent jump}
\end{figure}

For any angle $\al\in[0,\pi]$, set $\mathbf{e}_\al=(\cos\al,\sin\al)$. If $\al_{j+1}-\al_j=\pi$, then $j=1$, $0=\al_1<\al_2=\pi$, $b=a_i$ for some $i\in\{3,...,N\}$, $h_v(\theta)$ has the form \eqref{2 jump}. Therefore, we only need to consider the case $\al_{j+1}-\al_j<\pi$. Define the direction
\begin{equation*}
\xi=\sigma_{ab}\mathbf{e}_{\al_j}+\sigma_{bc}\mathbf{e}_{\al_{j+1}},
\end{equation*}
which lies strictly between $\mathbf{e}_{\al_j}$ and $\mathbf{e}_{\al_{j+1}}$. On the unit half-disk $B_1^+$, denote the sector of $b$-phase by
\[
S:=\{(r\cos\theta,r\sin\theta):r\in(0,1), \theta\in(\al_j,\al_{j+1})\}.
\]
We next perturb the interface $\{r\mathbf{e}_{\al_j}: r\in(0,1)\}\cup\{r\mathbf{e}_{\al_{j+1}}: r\in(0,1)\}$ in $S$ by pinching together $a$-phase and $c$-phase along the $\mathbf{\xi}$ direction, thus creating a triple junction of $a,b,c$ inside $S$. More precisely, for $0<t\ll 1$, we replace $\{r\mathbf{e}_{\al_j}: r\in(0,1)\}\cup\{r\mathbf{e}_{\al_{j+1}}: r\in(0,1)\}$ by the union of segments joining $t\f{\mathbf{\xi}}{|\mathbf{\xi}|}$ to $\mathbf{e_{\al_j}}$ and $\mathbf{e}_{\al_{j+1}}$, and a length-$t$ $a$-$c$ interface joining $t\f{\xi}{|\xi|}$ to the origin. In case $a=c$, omit the last small $a$-$c$ interface. We compute the energy change on $B_1^+$
\begin{align*}
    \Delta J_*&=\left(\left|t\f{\xi}{|\xi|}-\mathbf{e_{\alpha_j}}\right|-1 \right)\sigma_{ab}+ \left(\left|t\f{\xi}{|\xi|}-\mathbf{e_{\alpha_{j+1}}}\right|-1 \right)\sigma_{bc}+ t\sigma_{ac}\\
    &=\f{-t}{|\xi|} \xi\cdot(\sigma_{ab}\mathbf{e}_{\al_j}+\sigma_{bc}\mathbf{e}_{\al_{j+1}})+t\sigma_{ac}+O(t^2)\\
    &=t(\sigma_{ac}-|\xi|)+O(t^2).
\end{align*}
By minimality we have 
\begin{equation*}
    \sigma_{ac}\geq |\sigma_{ab}\mathbf{e}_{\al_j}+\sigma_{bc}\mathbf{e}_{\al_{j+1}}|\ \Rightarrow \ \sigma_{ac}^2\geq \sigma_{ab}^2+\sigma_{bc}^2 +2\cos(\al_{j+1}-\al_{j})\sigma_{ab}\sigma_{bc},
\end{equation*}
Since cosine function is decreasing on $(0,\pi)$, we have 
\begin{equation}\label{est: one sector angle}
    \al_{j+1}-\al_j\geq \theta_b= \arccos\left(\f{\sigma_{ac}^2-\sigma_{ab}^2-\sigma_{bc}^2}{2\sigma_{ab}\sigma_{bc}}\right),
\end{equation}
where $\theta_b$ is the angle of the $b$-phase at a triple junction of phases $a,b,c$ determined by \eqref{young's law}. This also shows $a\neq c$ when $\al_{j+1}-\al_j<\pi$.

Under ($\mathrm{H}_3$), i.e. $N=3$, we first claim that $k\leq 2$. Otherwise assume $k\geq 3$, by the pinching argument, there exist $i,j \in\{1,2,3\},\,i\neq j$, such that 
\begin{equation*}
\al_3-\al_1=(\al_3-\al_2)+(\al_2-\al_1)\geq \theta_i+\theta_j>\pi,
\end{equation*}
which yields a contradiction. When $k=2$, the middle phase must be $a_3$, which verifies the form \eqref{2 jump} with the angle condition $\al_2-\al_1\geq \theta_3$. When $k=1$, there is only one $a_1$-$a_2$ phase transition, which corresponds to \eqref{1 jump}.

Under ($\mathrm{H}_3'$), we also assert that $k\leq 2$, since otherwise there will be at least two middle phases, each with an opening angle greater than or equal to $\f{2\pi}{3}$. Then $k=2$ gives \eqref{2 jump} with the third phase index $i\in\{3,...,N\}$ and $k=1$ gives the form \eqref{1 jump}. 

Next we show that the partitions in (1) and (2) are minimizing. First we consider the 1-jump case \eqref{1 jump}. Let $K\subset \BR_+^2$ be a bounded Lipschitz domain and $\mathbf{\nu}_\al:=(-\sin\alpha,\cos\alpha)$. Consider a competitor $w=\sum_i a_i\chi_{E_i}$ agreeing with $v$ outside $K$ and on $\pa K\cap \BR_+^2$ in the sense of trace. Thus $\mathrm{tr}_{\pa K\cap\BR_+^2} \,w$ only takes values $a_1$ and $a_2$.  Denote by $\nu_{ij}$ the unit normal on $\pa^*E_i\cap\pa^*E_j$ pointing from $E_i$ to $E_j$ and set
\begin{align*}
  &\qquad\xi_1:=0,\quad \xi_2:= -\sigma_{12}\nu_\alpha,\\
  & \begin{cases}
      \xi_3:= -\f{\sigma_{13}\sigma_{12}}{\sigma_{13}+\sigma_{23}}\nu_\alpha, & \text{under }(\mathrm{H}_3),\\
      \xi_i:= -\f{\sigma_{12}}{2}\nu_\al,\ \forall\,i\geq 3, & \text{under }(\mathrm{H}_3').
  \end{cases}
\end{align*}
Then it is straightforward to verify that 
\begin{equation*}
    |\xi_1-\xi_2|=\sigma_{12},  \qquad |\xi_i-\xi_j|<\sigma_{ij}\ \text{ for any }\{i,j\} \neq \{1,2\}.
\end{equation*}
Using the divergence theorem we have 
\begin{equation}\label{calibration}
\begin{split}
    &J_*(w,K)\geq \sum_{i<j} \int_{K\cap \pa^*E_i\cap\pa^*E_j} (\xi_i-\xi_j)\cdot \nu_{ij}\,d\ch\\
    &\qquad\quad+\int_{\pa K\cap \pa \BR_+^2\cap\{x>0\}} (\xi_{\mathrm{tr} w}-\xi_1)\cdot\nu_{\pa \BR_+^2}\,d\ch +\int_{\pa K\cap \pa \BR_+^2\cap\{x<0\}} (\xi_{\mathrm{tr} w}-\xi_2)\cdot\nu_{\pa \BR_+^2}\,d\ch \\
    &\qquad= -\int_{\pa K} \xi_{\mathrm{tr}\,w}\cdot \nu_K \,d\ch+\int_{\pa K\cap\pa\BR_+^2}(\xi_{\mathrm{tr}\,w}-\xi_g)\cdot \nu_{\pa \BR_+^2}\,d\ch\\
    &\qquad=-\int_{\pa K\cap \BR_+^2}\xi_{\mathrm{tr}\,w}\cdot \nu_K \,d\ch-\int_{\pa K\cap \BR_+^2} \xi_g\cdot \nu_{\pa \BR_+^2}\,d\ch.
    \end{split}
\end{equation}
The integrals on the last line only depends on the trace on $\pa K\cap \BR_+^2$ and the boundary condition $g=a_1\chi_{x>0}+a_2\chi_{x<0}$. For the partition $v$ of the form \eqref{1 jump}, since $(\xi_1-\xi_2)\cdot \nu_{12}=\sigma_{12}$ and there are only two phases $a_1,\,a_2$, the inequality becomes equality in \eqref{calibration} for $v$, which implies the minimality.

For the two-jump case (2) we follow the same idea. Under ($\mathrm{H}_3$), set 
\begin{equation}\label{cal vector}
    \xi_3=0,\quad \xi_1=\sigma_{13}\nu_{\al_1},\quad \xi_2=-\sigma_{23} \nu_{\al_2}. 
\end{equation}
One immediately has 
\[
|\xi_1-\xi_3|= (\xi_1-\xi_3)\cdot \nu_{13}=\sigma_{13},\quad |\xi_2-\xi_3|= (\xi_2-\xi_3)\cdot\nu_{23}= \sigma_{23}.
\]
Moreover, 
\[
|\xi_1-\xi_2|=\sqrt{\sigma_{13}^2+\sigma_{23}^2+2\sigma_{13}\sigma_{23}\cos(\al_2-\al_1)}\leq \sqrt{\sigma_{13}^2+\sigma_{23}^2+2\sigma_{13}\sigma_{23}\cos(\theta_3)}=\sigma_{12}.
\]
Thus the same computation as in \eqref{calibration} can be carried out for this new set of $\{\xi_i\}$ and equality holds for $v$ of the form \eqref{2 jump}, which proves the minimality. Under ($\mathrm{H}_3'$), simply replace the index $3$ in \eqref{cal vector} by the middle phase index $i$. For any unused index $j$, just assign $\xi_j=0$. Then the proof still works. 

\end{proof}

\begin{corol}\label{corol:limit density}
    By Proposition \ref{prop:homogenity of blowdown} and \ref{prop: cla of blowdown}, for the minimizing solution $u$, it holds that 
    \begin{equation*}
        \lim_{r\to\infty} \f{J(u,B_r^+)}{r}=\begin{cases}
            \sigma_{12},&  \text{some blow-down has the form \eqref{1 jump}},\\
            \sigma_{13}+\sigma_{23}, & \text{under ($\mathrm{H}_3$), some blow-down has the form \eqref{2 jump}},\\
            2\sigma, & \text{under ($\mathrm{H}_3'$), some blow-down has the form \eqref{2 jump}}.
        \end{cases}
    \end{equation*}
    In particular, all sequential blow-downs belong to the same family, either the one-jump case (1) or the two-jump case (2). In addition, under ($\mathrm{H}_3'$), the third phase $a_i$ ($i\neq 1,2$) for the two-jump solution remains the same for all blow-downs.  
\end{corol}

\begin{proof}
The classification of possible energy densities follows directly from Proposition \ref{prop: cla of blowdown}. It suffices to prove the last statement: the third phase stays the same for all possible sequential blow-downs, given a blow-down limit has the form \eqref{2 jump} under ($\mathrm{H}_3'$). 

Since $\al_2-\al_1\geq \frac{2\pi}{3}$, we have for any $r_j\to\infty$, after passage to a subsequence, there exists $i\in \{3,...,N\}$ such that  
\begin{equation*}
    \|u_{r_j} -a_i\|_{L^1 (K)}\ \rightarrow\ 0,
\end{equation*}
where $K$ denotes the sector 
\begin{equation*}
    K:=\{(r\cos\theta,r\sin\theta):r\in(0,1),\ \theta\in(\f{\pi}{3},\f{2\pi}{3})\}.
\end{equation*}
Consequently, 
\begin{equation*}
    \lim\limits_{r\to\infty} \dist(\mathfrak{u}(r), \{a_3,...,a_N\})=0,
\end{equation*}
where $\mathfrak{u}(r):= \fint_{K} u_r(z)\,dz $. Note that $\mathfrak{u}(r)$ is continuous in $r$. Hence, one can identify a unique $i\in\{3,...,N\}$ such that $\lim\limits_{r\to\infty} \mathfrak{u}(r)=a_i$, which implies the uniqueness of the third phase for all sequential blow-downs. 
    
\end{proof}

\section{Uniqueness of the blow-down limit}\label{sec:uniqueness}

In this section we show that the blow-down limit $\tilde{u}$ is independent of the choice of sequence $\{r_j\}$. The proof applies to both ($\mathrm{H}_{3}$) and ($\mathrm{H}_{3}'$). We will utilize the convergence of energy density in Proposition \ref{prop: compactness of blow-down map} and the classification of blow-down limits in Section \ref{sec:classification}.

Define the stress-energy tensor
    \begin{equation*}
        T_{ij}:= \pa_{i} u\cdot \pa_{j}u -\delta_{ij}\left(\f12|\na u|^2+W(u)\right),\quad i,j\in\{1,2\}.
    \end{equation*}
Then $\mathrm{div}(T_i)=\pa_{j}T_{ij}= \pa_{i} u\cdot(\Delta u-\nabla_u W(u))=0$ by \eqref{eq:2D allen cahn}. We also define the boundary flux on the semicircle $\pa B_r\cap \BR_+^2$ as
\begin{equation*}
    \mathcal{F}(r):= -\int_{\pa B_r\cap \BR_+^2} T\,\nu\,d\ch,
\end{equation*}
where $\nu$ is the outer unit normal on $\pa B_r\cap \BR_+^2$. 

\begin{lemma}\label{lem: bdy flux}
    The inward normal derivative $\pa_y u(\cdot, 0)$, which we abbreviate as $\pa_y u$ below, belongs to $L^2(\BR;\BR^2)$. Moreover,
    \begin{equation}
        \label{conv bdy flux} \lim\limits_{r\to+\infty} \mathcal{F}(r)= \left( -\int_{\BR} u_0'\cdot \pa_yu\,dx,\ \int_\BR (\f12|u_0'|^2-\f12|\pa_y u|^2+W(u_0))\,dx \right)=: \widetilde{\mathcal{F}}.
    \end{equation}
    If a $L^1$ sequential blow-down limit $\tilde{u}$ has jump angle(s) $\al_l\in[0,\pi]$ (either one jump angle as in \eqref{1 jump} or two jump angles as in \eqref{2 jump}), with associated weight $\sigma_l$, then 
    \begin{equation}\label{tcf}
        \tcf=\sum_l \sigma_l \mathbf{e}_{\al_l}.
    \end{equation}
\end{lemma}

\begin{proof}
    Since $\mathrm{div}(T)=0$, the divergence theorem implies 
    \begin{equation}\label{flux div thm}
        \mathcal{F}(r)=\int_{-r}^r T\,(0,-1)^T\,dx=\left(-\int_{-r}^r u_0'\cdot \pa_y u\,dx,\ \int_{-r}^r\big(\f12|u_0'|^2-\f12|\pa_y u|^2+W(u_0)\big)\,dx \right).
    \end{equation}

Let $\mathcal{F}_2(r)$ denote the second component of $\mathcal{F}(r)$. Since 
\begin{align*}
    &|\mathcal{F}_2(r)|\leq \int_{\pa B_r\cap\BR_+^2} \sqrt{T_{21}^2+T_{22}^2}\,d\ch\\
    &\qquad \leq \int_{\pa B_r\cap\BR_+^2} \sqrt{(\f12|\na u|^2+W(u))^2-2|\pa_y u|^2 W(u)}\,d\ch\\
    &\qquad \leq  \int_{\pa B_r\cap\BR_+^2} \left(\f12|\na u|^2+W(u)\right)\,d\ch = \f{d\, J(u,B_r^+)}{dr},\quad \text{a.e. }r>0.
\end{align*}
By Lemma \ref{lem:rough ene upper bdd}, $J(u,B_r^+)\leq Cr$. One can find a sequence $R_k\to\infty$ such that 
\[
|\mathcal{F}_2(R_k)|\leq \f{d\, J(u,B_{r}^+)}{dr}\big|_{r=R_k}\leq 2C. 
\]
\eqref{flux div thm} and \eqref{u0 finite energy} then imply that
\begin{equation*}
    \int_{-R_k}^{R^k} \f12|\pa_y u(\cdot,0)|^2\,dx \leq 2C+\int_{-R_k}^{R_k} \big(\f12|u_0'|^2+W(u_0)\big)\,dx<\infty,
\end{equation*}
which further implies $\pa_y(u)(\cdot,0)\in L^2(\BR;\BR^2)$. \eqref{conv bdy flux} follows immediately from \eqref{flux div thm} by taking $r\to\infty$. 

To prove \eqref{tcf}, we fix the sequence $r_j\to\infty$ such that $u_{r_j}\to \tu$ in $L_{loc}^1$. Take a cutoff function $\eta\in C_c^\infty((1,2);[0,\infty))$ such that $\int_1^2 \eta(r)\,dr=1$. We have
\begin{equation}\label{identify tsf 1}
\begin{split}
    \int_1^2\eta(r)\mathcal{F}(r_jr)\,dr &=\f{1}{r_j}\int_{r_j}^{2r_j} \eta(\f{r}{r_j}) \int_{\pa B_r\cap \BR_+^2} \left(-(\na u)^T \pa_r u+ (\f12|\na u|^2+W(u))\nu\right) \,d\ch dr\\
    &= \f{1}{r_j}\int_{A_{r_j,2r_j}^+ }\eta(\f{|z|}{r_j})\left(-(\na u)^T \pa_r u+ (\f12|\na u|^2+W(u))\nu\right)\,dz.
\end{split}
\end{equation}
By \eqref{lim: small rad der} and Cauchy-Schwarz, the first term can be controlled by 
\begin{equation}\label{identify tsf 2}
    \f{\|\eta\|_{L^\infty}}{r_j} \left(\int_{A_{r_j,2r_j}^+ } |\na u|^2\,dz\right)^{1/2}\left( \int_{A_{r_j,2r_j}^+ } |\pa_r u|^2\,dz \right)^{1/2}\ \longrightarrow \ 0, \quad j\to\infty.
\end{equation}
Moreover, rescaling and Proposition \ref{prop: compactness of blow-down map} (in particular, \eqref{comp for measure}) imply
\begin{equation}\label{identify tsf 3}
    \begin{split}
        &\f{1}{r_j}\int_{A_{r_j,2r_j}^+ }\eta(\f{|z|}{r_j})\left(\f12|\na u|^2+W(u)\right)\nu\,dz\\
    =& \int_{A_{1,2}^+} \eta(|z|) \left(\f{1}{2r_j}|\na u_{r_j}|^2 +r_j W(u_{r_j})\right)\f{z}{|z|}\,dz\\
    \to & \int_{\BR_+^2} \eta(|z|) \f{z}{|z|} \,d \mu_{\tu}= \int_{\BR_+^2} \eta(|z|) \f{z}{|z|} d\left(\sum_l \sigma_l \ch \mres \l_{\al_l}\right)=\sum_{l} \sigma_l\,\mathbf{e}_{\al_l}.
    \end{split}
\end{equation}
Note that the boundary mismatch has been counted, as the jump angle $\al_l$ could be $0$ or $\pi$. Finally, we can derive \eqref{tcf} directly from \eqref{conv bdy flux} and \eqref{identify tsf 1}--\eqref{identify tsf 3}. 

\end{proof}

\begin{proposition}
    \label{prop: uniquenss of blowdown}
    There exists a unique blow-down limit $\tilde{u}$, with the structure characterized by Proposition \ref{prop: cla of blowdown}, such that 
    \begin{equation*}
        \lim\limits_{r\ri\infty} \|u_r-\tilde{u}\|_{L^1(K)}=0,
    \end{equation*}
    for any compact $K\subset \overline{\BR^2_+}$.
\end{proposition}

\begin{proof}
    By Corollary \ref{corol:limit density}, it suffices to show that the jump angles are uniquely determined for all blow-downs. If there is only one jump, corresponding to the form \eqref{1 jump}, then \eqref{conv bdy flux} and \eqref{tcf} imply $\lim_{r\to\infty} \mathcal{F}(r)=\sigma_{12}\mathbf{e}_{\al}$, thus $\al$ is uniquely determined. 

    For the two-jump profile \eqref{2 jump}, $\widetilde{\mathcal{F}}= \sigma_{1i}\mathbf{e}_{\al_1}+\sigma_{i2}\mathbf{e}_{\al_2}$. Assume either ($\mathrm{H_3}$) or ($\mathrm{H_3'}$), and suppose there exists $(\tilde{\al}_1,\tilde{\al}_2)\neq (\al_1,\al_2)$ such that 
\[
\sigma_{1i}\mathbf{e}_{\al_1}+\sigma_{i2}\mathbf{e}_{\al_2}= \sigma_{1i}\mathbf{e}_{\tilde{\al}_1}+\sigma_{i2}\mathbf{e}_{\tilde{\al}_2}.
    \]
    Then $\al_2-\al_1=\tilde{\al}_2-\tilde{\al}_1$ due to 
    \begin{equation*}
        |\widetilde{\mathcal{F}}|^2=\sigma_{1i}^2+\sigma_{i2}^2+2\sigma_{1i}\sigma_{i2}\cos(\al_2-\al_1)= \sigma_{1i}^2+\sigma_{i2}^2+2\sigma_{1i}\sigma_{i2}\cos(\tilde{\al}_2-\tilde{\al}_1)
    \end{equation*}
    Therefore, $\mathbf{e}_{\al_1}-\mathbf{e}_{\tilde{\al}_1}$ and $\mathbf{e}_{\al_2}-\mathbf{e}_{\tilde{\al}_2}$ are two parallel non-zero chords on $B_1^+$ with the same opening angle, which yields a contradiction.  
    
\end{proof}

Now we are ready to conclude the proof of Theorem \ref{main thm: cls of blow down}.

\begin{proof}[Proof of Theorem \ref{main thm: cls of blow down}]
    The statement in Theorem \ref{main thm: cls of blow down} follows immediately from Proposition \ref{prop: compactness of blow-down map}, \ref{prop: cla of blowdown} and \ref{prop: uniquenss of blowdown}. 
\end{proof}

\section{Asymptotic behavior near the sharp interface}\label{sec:sharp interface}

This section is devoted to the proof of Theorem \ref{main thm: refined behavior near interface}, which establishes the asymptotic behavior of $u$ near the interior sharp interface. Throughout this section we assume ($\mathrm{H}_1$), ($\mathrm{H}_2$), either ($\mathrm{H}_3$) or ($\mathrm{H}_3'$), and ($\mathrm{H}_4$).

Let $u$ be a minimizing solution of \eqref{eq:2D allen cahn} satisfying \eqref{bdy cond}--\eqref{uniform bound}, and let $\tilde{u}(r,\theta)=\tilde{u}(\theta)$ be its unique blow-down limit of $u$ at infinity, i.e. for any bounded domain $\Om\subset \BR_+^2$,
\[
\lim\limits_{r\to\infty} \|u(r\cdot)-\tu\|_{L^1(\Om)}=0.
\]
In particular, we assume $\tilde{u}$ has the following structure (the intermediate phase is relabeled as $a_3$):
\begin{equation}\label{particular hatu}
    \tilde{u}(\theta)=\begin{cases}
    a_1, & \theta\in(0,\al_1)\\
    a_3, & \theta\in(\al_1,\al_2),\\
    a_2, & \theta\in(\al_2,\pi),
    \end{cases}
    \quad\text{ for some }0<\al_1< \al_2< \pi,\ \al_2-\al_1\geq \begin{cases}
        \theta_3,& (\mathrm{H}_3)\\
        \f{2\pi}{3}, & (\mathrm{H}_3').
    \end{cases}.
\end{equation}
The two discontinuity angles are $\alpha_1$ and $\alpha_2$, respectively. In the two-phase case where $\alpha_1 = \alpha_2 \in (0, \pi)$, the proof becomes simpler and is omitted. 

We now set the angles
\begin{equation}\label{def:theta i}
    \varphi_1:=\f{\al_1}{2}, \qquad \varphi_3:=\f{\al_1+\al_2}{2},\qquad \varphi_2:= \f{\al_2+\pi}{2},
\end{equation}
where for each $i$, $\mathbf{e}_{\varphi_i}$ represents the direction of the bisector of $a_i$ phase. Set $l_{\varphi_i}:=\{r\mathbf{e}_{\varphi_i}:r\in(0,\infty)\}$ as the ray emanating from the origin and in direction $\varphi_i$. We first establish the following exponential decay estimate in the interior of phase regions away from the sharp interface.

\begin{lemma}
    \label{lem:exponential decay phase}
    Take an open domain $\Om\subset\BR_+^2$ satisfying $\tu\equiv a_i$ in $\Om$. For every $K\Subset \Om$, there exist $c,\, C,\, r_0$, depending on $W,\,u,\,K$, such that for every $r>r_0$,
    \begin{equation}
        \label{exp decay compact set}
        \sup_{z\in K} \left( |u(rz)-a_i|+|\na u(rz)|+|\na^2 u(rz)| \right)\leq Ce^{-cr},
    \end{equation}
    where $\na$, $\na^2$ mean derivatives of $u$ evaluated at $rz$. In particular, on the rays $l_{\varphi_i}$ for $i=1,2,3$, 
    \begin{equation}\label{exp decay theta i}
        |u(r\mathbf{e}_{\varphi_i})-a_i|+|\na u(r\mathbf{e}_{\varphi_i})|+|\na^2 u(r\mathbf{e}_{\varphi_i})|\leq Ce^{-cr}\quad\text{for all sufficiently large }r. 
    \end{equation}
\end{lemma}

\begin{proof}
 Take $\rho>0$ such that $\{z: \dist(z,K)\leq \rho\}\subset \Om$. Then we have 
 \[
 \fint_{B_{r\rho} (r\xi)}|u(z)-a_i|\,dz=\fint_{B_\rho(\xi)}|u_r(z)-a_i|\,dz\, \to\, 0 \text{ uniformly for }\xi\in K.
 \]
Applying Lemma \ref{lem: clear out} gives \eqref{exp decay compact set}. 

To show \eqref{exp decay theta i}, simply choose $\Omega$ to be the $a_i$ sector for $\tu$, and let $K$ denote the disk $B_{r_i}(\mathbf{e}_{\varphi_i})$, where $r_i>0$ such that $B_{2r_i}(\mathbf{e}_{\varphi_i})\subset \Omega$. Then applying \eqref{exp decay compact set} to such choice of $\Om$ and $K$ gives \eqref{exp decay theta i}.

\end{proof}

The following lemma provides energy lower bound and upper bound on $B_R^+$.

\begin{lemma}\label{lem:sharp ene bd}
    There is a constant $C=C(u,W)$ such that for every $R>0$, 
    \begin{equation}\label{ineq:sharp bd}
        (\sigma_{13}+\sigma_{23}) R-C\leq J(u,B_R^+)\leq (\sigma_{13}+\sigma_{23}) R+C.
    \end{equation}
\end{lemma}

\begin{proof}
Define 
\begin{equation*}\label{def Ar}
    A(r):=\int_{\pa B_r\cap\BR_+^2} \left(\f12|\pa _T u|^2+W(u)\right)\,d\ch.
\end{equation*}
We apply the estimate \eqref{exp decay theta i} on the rays $l_{\varphi_i}$, and Lemma \ref{lemma: 1D energy estimate} to obtain 
\begin{align*}
     &J(u,B_R^+)\geq \int_0^R A(r)\,dr\\
    &\quad\geq  \int_0^R \left(\sigma_{13}+\sigma_{23}- C|u(r\mathrm{e}_{\varphi_1})-a_1|^2- C|u(r\mathrm{e}_{\varphi_3})-a_3|^2- C|u(r\mathrm{e}_{\varphi_2})-a_2|^2\right)\,dr\\
    &\quad\geq (\sigma_{13}+\sigma_{23}) R- C\int_0^R Ce^{-2cr}\,dr
    \geq (\sigma_{13}+\sigma_{23}) R-C(u,W),
\end{align*}
which yields the lower bound.

For the upper bound, we present a constructive proof. Fix $\g\leq \d_W$. Choose $0<\delta<\g$ and $\e>0$ to be chosen below. From \eqref{homog 1} we have 
\[
\lim_{R\to\infty} \f1R \int_R^{2R} A(r)\,dr =\sigma_{13}+\sigma_{32}. 
\]
Thus, for sufficiently large $R$, there exists $R_0\in (R,2R)$ such that 
\begin{equation}\label{good R0}
A(R_0)\leq \sigma_{13}+\sigma_{32}+\e. 
\end{equation}
For such $R_0$, there exist
\begin{align*}
T_{13}:=(\al_{13}^-,\al_{13}^+) \subset (\f{3\al_1}{4},\f{3\al_1+\al_2}{4}),\qquad T_{32}:=(\al_{32}^-,\al_{32}^+) \subset (\f{3\al_2+\al_1}{4},\f{3\al_2+\pi}{4}),
\end{align*}
satisfying 
\begin{align*}
    &|u(R_0\mathbf{e}_{\theta})-a_1|\leq \delta,\qquad \theta\in (\varphi_1,\al_{13}^-),\\
    &|u(R_0\mathbf{e}_{\theta})-a_3|\leq \delta \qquad \theta\in (\al_{13}^+,\al_{32}^-),\\
    &|u(R_0\mathbf{e}_{\theta})-a_2|\leq \delta \qquad \theta\in (\al_{32}^+,\varphi_2),
\end{align*}
thanks to the uniform convergence away from the sharp interface \eqref{exp decay compact set}. We further define
\begin{align*}
\beta_{13}^-:= \sup\{\theta\in \overline{T_{13}}: |u(R_0\mathbf{e}_\theta)-a_1|\leq \delta\},\qquad \beta_{13}^+:= \inf\{\theta\in \overline{T_{13}}: |u(R_0\mathbf{e}_\theta)-a_3|\leq \delta\},\\
\beta_{32}^-:= \sup\{\theta\in \overline{T_{32}}: |u(R_0\mathbf{e}_\theta)-a_3|\leq \delta\},\qquad \beta_{32}^+:= \inf\{\theta\in \overline{T_{32}}: |u(R_0\mathbf{e}_\theta)-a_2|\leq \delta\}.
\end{align*}
We claim the following hold given we choose suitable $\g,\,\delta,\,\e$.
\begin{align}
    \label{relation of beta+-} &\quad0<\beta_{13}^+-\beta_{13}^-<\f{C_\delta}{R_0},\qquad 0<\beta_{32}^+-\beta_{32}^-<\f{C_\delta}{R_0};\\
    \label{delta dist}&|u(R_0\mathbf{e}_{\theta})-a_j|>\delta \quad\text{ for any }j,\ \ \theta\in  (\beta_{13}^-,\beta_{13}^+)\cup (\beta_{32}^-,\beta_{32}^+);\\
    \label{gamma close phase} & \qquad u(R_0\mathbf{e}_{\theta}) \in \begin{cases}B_{\g}(a_1), &\theta \in (0,\beta_{13}^{-}),\\
    B_{\g}(a_3), & \theta \in (\beta_{13}^{+}, \beta_{32}^-),\\
    B_{\g}(a_2), & \theta \in (\beta_{32}^{+},\pi).
    \end{cases}
\end{align}

\begin{figure}[htbp]
\centering
\begin{tikzpicture}[ font=\small,line cap=round,line join=round]

\draw[black!50, densely dotted, line width=0.8pt] (58:1.3) -- (58:5.4);

\draw[red!20, line width=9pt](12:5.4)arc[start angle=12,end angle=49, radius=5.4];

\draw[red!50, line width=9pt](12:5.4)arc[start angle=12,end angle=34, radius=5.4];

\draw[blue!20, line width=9pt](68:5.4)arc[start angle=68,end angle=106, radius=5.4];

\draw[blue!50, line width=9pt](85:5.4)arc[start angle=85,end angle=106, radius=5.4];

\draw[black!85, line width=0.9pt](12:5.4)arc[start angle=12,end angle=106, radius=5.4];

\fill (34:5.4) circle[radius=1.6pt] node[anchor=west, xshift=5pt, yshift=4pt] {$\alpha_{13}^{-}$};

\fill (49:5.4) circle[radius=1.6pt] node[anchor=west, xshift=4pt, yshift=7pt] {$\beta_{13}^{-}$};

\fill (68:5.4) circle[radius=1.6pt] node[anchor=south, yshift=6pt]  {$\beta_{13}^{+}$};

\fill (85:5.4) circle[radius=1.6pt] node[anchor=south, yshift=5pt] {$\alpha_{13}^{+}$};

\draw[black!65, line width=0.6pt,decorate,decoration={brace, amplitude=4pt}] (49:5) -- (68:5);

\node[fill=white, inner sep=2pt] at (58:4.4) {$\le C_\delta$};

\node[text=blue] at (-0.8,6) {$|u-a_3|<\delta$};
\node[text=blue] at (1.4,6.5) {$|u-a_3|<\gamma$};
\node[anchor=west] at (3.1,5.4) {$|u-a_j|>\delta\quad\forall j$};
\node[anchor=west, text=red] at (4.6,4.1) {$|u-a_1|<\gamma$};
\node[anchor=west, text=red] at (5.5,2.2) {$|u-a_1|<\delta$};
\node[anchor=east, xshift=-6pt, text=black] at (58:2.2) {$l_{\alpha_1}$};
\node[anchor=west, text=black] at (1.6,2.4) {$a_1$--$a_3$ interface};
\node[anchor=north east, text=black] at (-1,4.9) {$\pa B_{R_0}$};

\end{tikzpicture}
\caption{Near $a_1$--$a_3$ transition on $\pa B_{R_0}\cap\BR_+^2$. The transition arc has length $R_0(\beta_{13}^{+}-\beta_{13}^{-})\le C_\delta$.}
\label{fig:transition-arc}
\end{figure}

Indeed, if $\beta_{13}^+<\beta_{13}^-$, then by definition for $\theta\in \overline{T_{13}}$, $u(R_0\mathbf{e}_\theta)$ visits $a_1$--$a_3$--$a_1$--$a_3$ phases in counterclockwise order. Thus $A(R_0)\geq 3\sigma_{13}+\sigma_{32}-o(1)$ which contradicts $A(R_0)<\sigma_{13}+\sigma_{32}+\e$. For the same reason, for $\theta\in (\beta_{13}^-,\beta_{13}^+)$, $u(R_0\mathbf{e}_{\theta})$ cannot visit the $\delta$-neighborhood of any additional phase $a_i$ simply because that will create an additional phase transition and contradict the upper bound. The same argument works for the $a_2$-$a_3$ interface. Thus we have verified \eqref{delta dist}. Consider the integral of $W$ over the small arc $\{R_0\mathbf{e}_{\theta}:\theta\in (\beta_{13}^-,\beta_{13}^+)\}$,
\[
C\geq \int_{\{R_0\mathbf{e}_{\theta}:\theta\in (\beta_{13}^-,\beta_{13}^+)\}} W(u(z))\,d\ch=  R_0\int_{\beta_{13}^-}^{\beta_{13}^+} W(u(R_0\mathbf{e}_\theta))\,d\theta\geq CR_0(\beta_{13}^+-\beta_{13}^-)\delta^2,
\]
which gives \eqref{relation of beta+-} after choosing $C_\delta=\frac{C}{\d^2}$. 

To show \eqref{gamma close phase}, it suffices to notice that the arcs $\{R_0\mathbf{e}_{\theta}:\theta\in (\beta_{13}^-,\beta_{13}^+)\}$ and $\{R_0\mathbf{e}_{\theta}:\theta\in (\beta_{32}^-,\beta_{32}^+)\}$ already carry at least energy $\sigma_{13}+\sigma_{32}-C\delta^2$. If outside these arcs on $\pa B_{R_0}\cap\BR_+^2$, $u$ leaves the $\g$-neighborhood of a phase, that will cause at least $c\g^2$ additional energy when the distance to a particular energy well increase from $\delta$ to $\g$, given that $\d\ll \g$. This leads to a contradiction with $A(R_0)\leq \sigma_{13}+\sigma_{32}+\e$, if we choose $\delta$ and $\e$ so that $\e+C\delta^2<c\g^2$. Thus we have verified \eqref{relation of beta+-}--\eqref{gamma close phase}. See Figure \ref{fig:transition-arc} for an illustration of the choice of angles and the behavior of $u$ near one interface.

Further define
\begin{align*}
 &\qquad \beta_{13}:=\f{\beta_{13}^-+\beta_{13}^+}{2},\qquad \beta_{32}:=\f{\beta_{32}^-+\beta_{32}^+}{2},\\
 & v(R_0\mathbf{e}_\theta):= a_1\chi_{(0,\beta_{13})}(\theta)+a_3\chi_{(\beta_{13},\beta_{32})}(\theta)+a_2\chi_{(\beta_{32},\pi)}(\theta),
\end{align*}
where $v\in \mathrm{BV}(\pa B_{R_0}\cap\BR_+^2;\{a_1,a_2,a_3\})$. Then \eqref{relation of beta+-}--\eqref{gamma close phase} imply that 
\begin{equation}\label{u R_0 L2 est}
    \|u-v\|_{L^2(\pa B_{R_0}\cap \BR_+^2)}\leq C(W,M),
\end{equation}
whenever $R$ (and therefore $R_0$) is sufficiently large. Also $\beta_{13}\to \al_1$ and $\beta_{32}\to \al_2$ as $R\to\infty$ due to Lemma \ref{lem:exponential decay phase}.

Recall that $U_{ij}$ denotes the heteroclinic connection connecting $a_i$ to $a_j$ defined in ($\mathrm{H}_2$).  Set
\begin{align*}
   &\quad\qquad\quad\nu_{13}=(-\sin\beta_{13},\cos\beta_{13}),\quad \nu_{32}=(-\sin\beta_{32},\cos\beta_{32}),\\
   & \mathcal{D}_{13}:= \{r\mathbf{e}_\theta: r\in (0,R_0),\,\theta\in(0,\varphi_3)\},\quad \mathcal{D}_{32}:= \{r\mathbf{e}_\theta: r\in (0,R_0),\,\theta\in(\varphi_3,\pi)\}.
\end{align*}   

For sufficiently large $R_0$, define
\begin{equation*}
    \tilde{V}(z):= U_{13}(z\cdot \nu_{13})+U_{32}(z\cdot \nu_{32})-a_3. 
\end{equation*}
We first estimate $J(\tilde{V},B_{R_0}^+)$. Since $\mathcal{D}_{13}$ is angularly separated from $l_{\beta_{32}}$, the exponential decay of $U_{32}$ implies 
\begin{equation*}
    |U_{32}(z\cdot \nu_{32})-a_3|+|\na U_{32}(z\cdot \nu_{32})|+|\na^2 U_{32}(z\cdot \nu_{32})|\leq Ce^{-c|z|},\quad z\in\mathcal{D}_{13}.
\end{equation*}
This implies
\begin{equation*}
    J(\tilde{V},\mathcal{D}_{13}) \leq J(U_{13}(z\cdot \nu_{13}),\mathcal{D}_{13})+C,
\end{equation*}
since the contribution from $U_{32}(z\cdot \nu_{32})-a_3$ is exponentially small and can be absorbed by a uniform constant independent of $R$. On $\mathcal{D}_{13}$, using the coordinates $z= s \mathbf{e}_{\beta_{13}}+t\nu_{13}$, the line segment parallel to $\mathbf{e}_{\beta_{13}}$ with fixed $t$ has length at most $R_0+|t\cot{\beta_{13}}|$. This implies 
\[
J(U_{13}(z\cdot \nu_{13}),\mathcal{D}_{13})\leq R_0 \sigma_{13}+ |\cot{\beta_{13}}|\int_\BR |t| \left(\f12|U_{13}'(t)|^2+W(U_{13}(t))\right)\,dt\leq R_0\sigma_{13}+C,
\]
where we used the fact that $\f12|U_{13}'(t)|^2+W(U_{13}(t))$ decays exponentially at both ends and therefore is integrable after multiplying $|t\cot(\beta_{13})|$. Consequently,
\[
J(\tilde{V},\mathcal{D}_{13})\leq \sigma_{13}R_0+C.
\]
A similar argument applies to $\mathcal{D}_{32}$, where the major contribution comes from $U_{32}(z\cdot \nu_{32})$, and leads to the estimate
\[
J(\tilde{V},B_{R_0}^+)\leq (\sigma_{13}+\sigma_{32})R_0+C,
\]
for some constant depending on $u,W$ and uniform for large $R$. 

Choose a fixed smooth cutoff $0\le\phi\le1$ with
$\phi(0)=0$ and $\phi(y)=1$ for $y\ge1$, and set
\[
\hat{V}(x,y):=
\phi(y)\tilde  V(x,y)+(1-\phi(y))u_0(x).
\]
Note that since $\hat{V}$ is uniformly bounded, ($\mathrm{H}_1$) implies $W(\hat{V}(z))\leq C(W,M) |\hat{V}-a_i|^2$ for any index $i$. We have 
\begin{align*}
    J(\hat{V},B_{R_0}^+\cap\{0<y<1\})&\leq  C\bigg[\int_{B_{R_0}^+\cap\{0<y<1\}} |\na \tilde{V}|^2\,dz+ \int_\BR |u_0'|^2 +|u_0-(a_1\chi_{x>0}+a_2\chi_{x<0})|^2\,dx \\
    &\qquad\quad+ \int_{B_{R_0}^+\cap\{0<y<1\}} |\hat{V}(x,y)-(a_1\chi_{x>0}+a_2\chi_{x<0})|^2\,dxdy\bigg]\,\\
    &\leq C(W,u)
\end{align*}
for sufficiently large $R_0$. The last inequality follows from \eqref{u0 finite energy}, \eqref{u_0 conv to a1 a2} and exponential decay of $\tilde{V}$ to the corresponding boundary phases. 

The construction also implies
\begin{align}
   \nonumber &\hat{V}(x,0)=u_0(x),\qquad J(\hat V,B_{R_0}^+)\le (\sigma_{13}+\sigma_{32})R_0+C,\\
\label{hat V in annulus}&J(\hat V, A_{R_0-1,R_0}^+)+\|\partial_T\hat{V}\|_{L^2(\pa B_{R_0}\cap\BR_+^2)}^2+
              \|\hat V-v\|_{L^2(\pa B_{R_0}\cap\BR_+^2)}^2\le C.
\end{align}

Define the final competitor by
\begin{equation*}
    V(r\mathbf{e}_\theta)=\begin{cases}
        \hat{V}(r\mathbf{e}_\theta), & 0<r<R_0-1,\\
        \hat{V}(r\mathbf{e}_\theta)+(r-R_0+1) (u(R_0\mathbf{e}_\theta)-\hat{V}(R_0\mathbf{e}_\theta)), & R_0-1\leq r\leq R_0.
    \end{cases}
\end{equation*}

Thus $V$ matches $u$ on both the wall
and the semicircle. We compute
\[
J(V,A_{R_0-1,R_0}^+)\leq C\left(J(\hat V,A_{R_0-1,R_0}^+)+\|u-\hat V\|_{L^2(\pa B_{R_0}\cap\BR_+^2)}^2+\|\partial_T(u-\hat V)\|_{L^2(\pa B_{R_0}\cap\BR_+^2)}^2\right)
\leq C,
\]
which further implies 
\begin{equation*}
    J(u,B_{R_0}^+)\leq J(V,B_{R_0}^+)\leq (\sigma_{13}+\sigma_{32})R_0+C. 
\end{equation*}
Finally, for sufficiently large $R$, choose such $R_0\in(R,2R)$, and use the lower bound on the annular region:
\[
J(u,B_{R}^+)=J(u,B_{R_0}^+)-J(u,A^+_{R,R_0})\leq (\sigma_{13}+\sigma_{32})R+C(W,u). 
\]
This completes the proof of the upper bound.

\end{proof}

Next we focus on the $a_1$-$a_3$ interface $l_{\al_1}$ and establish \eqref{main thm 2:convergence to Uij} for $\al=\al_1$ and $(i,j)=(1,3)$. The same argument holds for $a_2$-$a_3$ interface $l_{\al_2}$. 

Consider the sectors
\begin{equation*}
    \mathcal{S}_{13}:= \{r\mathbf{e}_\theta: r>0,\,\varphi_1<\theta<\varphi_3\},\qquad \mathcal{S}_{32}:= \{r\mathbf{e}_\theta: r>0,\,\varphi_3<\theta<\varphi_2\}
\end{equation*}
where the angles are defined in \eqref{def:theta i}. Using Lemma \ref{lem:sharp ene bd} we obtain that 
\begin{equation}
    \label{ene bd S13} \sigma_{13}R-C\leq J(u,\mathcal{S}_{13}\cap B_R)\leq J(u,B_R^+)-J(u,\mathcal{S}_{32}\cap B_R)\leq \sigma_{13}R+C,
\end{equation}
\begin{equation}\label{est: dr S13}
\begin{split}
&\f12\int_{\mathcal{S}_{13}\cap \{|z|>r_0\}}|\pa_ru|^2\,dz \\
\leq &\lim_{R\to\infty}\left[J(u,\mathcal{S}_{13}\cap A_{r_0,R})-\int_{r_0}^R \int_{\mathcal{S}_{13}\cap\pa B_r}\left( \f12|\pa_T u|^2+W(u) \right)\,d\ch\,dr\right]\\
\leq &\lim_{R\to\infty}(\sigma_{13}R+C-\int_{r_0}^R(\sigma_{13}-Ce^{-cr})\,dr)\leq C,
\end{split}
\end{equation}
where the lower bounds of $J(u,\mathcal{S}_{13}\cap B_R)$, $J(u,\mathcal{S}_{32}\cap B_R)$, and the tangential energy cost in $\mathcal{S}_{13}$ follow from the exponential decay estimate on the boundary rays \eqref{exp decay theta i}.

Recall the stress-energy tensor
\begin{equation*}
        T_{ij}:= \pa_{i} u\cdot\pa_{j}u -\delta_{ij}\left(\f12|\na u|^2+W(u)\right),\qquad \mathrm{div}\, T=0.
    \end{equation*}
Set 
\[
\mathcal{F}_{13}(r):= \int_{\pa B_r\cap\mathcal{S}_{13}} -T\nu\,d\ch.
\]
The stress tensor $T$ on $\pa \mathcal{S}_{13}=l_{\varphi_1}\cup l_{\varphi_3}$ is $O(e^{-cr})$ by Lemma \ref{lem:exponential decay phase}. Thus, for any $ t>s>r_0$,
\begin{equation}\label{eq:cauchy conv} |\mathcal{F}_{13}(t)-\mathcal{F}_{13}(s)|\leq C\int_s^t e^{-cr}\,dr\leq Ce^{-cs}.
\end{equation}
Thus $\mathcal{F}_{13}(r)$ forms a Cauchy sequence and has a limit as $r$ tends to $\infty$. We have 
\begin{equation*}
\f1R\int_R^{2R}\mathcal{F}_{13}(r)\,dr=\frac1R\int_{\mathcal{S}_{13}\cap A_{R,2R}}\left(-(\na u)^T \pa_r u+ (\f12|\na u|^2+W(u))\nu\right)\,dz
\end{equation*}
We can apply the estimates in \eqref{identify tsf 2} and \eqref{identify tsf 3} to identify the limit
\begin{equation}\label{id lim F13}
 \lim_{r\to\infty}\mathcal{F}_{13}(r)= \sigma_{13}\mathbf{e}_{\al_1},
\end{equation}
which together with the convergence rate 
\eqref{eq:cauchy conv} gives
\begin{equation}\label{id lim F13 2}
\mathcal{F}_{13}(r)= \sigma_{13}\mathbf{e}_{\al_1}+O(e^{-cr}).
\end{equation}

Consider an orthogonal change of coordinates $z=x\mathbf{e}_{\al_1}+y \nu_{\al_1}$ ($\nu_{\al_1}=\mathbf{e}_{\al_1}^\perp$) and define 
\begin{equation*}
    v(x,y) = u(x\mathbf{e}_{\al_1}+y \nu_{\al_1})=u(x\cos\al_1 - y\sin\al_1,\, x\sin\al_1 + y\cos\al_1).
\end{equation*}

Set
\[
k_1:=\tan{\f{\al_1}{2}},\qquad  k_3:=\tan{\f{\al_2-\al_1}{2}},\qquad \mathcal{S}=\{(x,y):x>0, -k_1x<y<k_3x\}.
\]
Then \eqref{exp decay theta i} becomes the following estimate on $\pa\mathcal{S}$,
\begin{equation}\label{decay on paS}
\begin{split}
    &|v(x, k_3x)-a_3| +|\na v(x, k_3x)|+|\na^2 v(x, k_3x)| \leq C e^{-cx},\\
    &|v(x, -k_1x)-a_1|+|\na v(x, -k_1x)|+|\na^2 v(x, -k_1x)| \leq C e^{-cx}, 
\end{split}
\end{equation}
for sufficiently large $x$. 

For convenience, we extend our definition of $v$ from $\mathcal{S}$ to the half-plane $\{x>0\}$. The extended function, still denoted by $v(x,y)$, satisfies
\begin{equation*}
    v(x,y)=\begin{cases}
        v(x,y), & (x,y)\in \mathcal{S},\\
        a_3,  &y> k_3x+1\\
        a_1,  &y<-k_1x-1\\ 
        \text{smooth interpolation between } a_3\,\text{ and }v(x, k_3x), &    k_3x<y< k_3x+1,\\
        \text{smooth interpolation between } a_1\,\text{ and }v(x, -k_1x), &     -k_1x-1<y<-k_1x.
    \end{cases}
\end{equation*}
By \eqref{decay on paS},
\begin{equation}\label{ene v negli}
 \int_{\BR\setminus [-k_1x,k_3x]} (|\pa_x v|^2+|\pa_y v|^2+W(v))\,dy \leq Ce^{-cx}, \quad \text{for large } x.  
\end{equation}

\begin{lemma}
 Set
    \begin{equation*}
       \Sigma_x:=\{(x,y):y\in \BR\},\quad  G(x):= \int_{\Sigma_x} \left(\f12|\pa_y v|^2-\f12|\pa_x v|^2+W(v)\right)\,dy, \quad H(x):= \int_{\Sigma_x} \pa_x v\cdot \pa_y v\,dy.
    \end{equation*}
    There exist constants $C,\,c,\,x_0\,>0$, depending on $W$ and $u$, such that for every $x>x_0$,  
    \begin{equation}\label{est: G H}
        \left|G(x) -\sigma_{13} \right|+|H(x)|\leq Ce^{-cx}.
    \end{equation}
    \begin{equation}\label{est: pa x}
        \int_{x_0}^\infty\int_{\S_x} |\pa_xv(x,y)|^2\,dydx< \infty.
    \end{equation}
    \begin{equation}\label{est: ene excess}
        \int_{x_0}^\infty\left(\int_{\S_x}  \left(\f12|\pa_y v(x,y)^2|+W(v(x,y))\right)\,dy-\sigma_{13}\right)\,dx<\infty.
    \end{equation}
\end{lemma}

\begin{proof}
First we prove \eqref{est: pa x}. By \eqref{ene v negli} it suffices to show 
\begin{equation*}
\int_{\mathcal{S}\cap\{x>x_0\}}|\pa_x v|^2\,dz <\infty   
\end{equation*}
for some $x_0>0$, which is equivalent to
\begin{equation}
\label{est du al1}   \int_{\mathcal{S}_{13}\cap\{r>r_0\}}|\pa_{\mathbf{e}_{\al_1}} u|^2\,dz<\infty, 
\end{equation}
for some $r_0>0$.

Utilizing $\mathbf{e}_{\al_1}\cdot \mathbf{e}_{\theta}>0$ for any $\theta\in(\varphi_1,\varphi_3)$, and the relation
\[
\f12(1-\mathbf{e}_{\al_1}\cdot \mathbf{e}_{\theta})|\na u|^2 +\pa_{\mathbf{e}_{\al_1}} u\cdot \pa_{\mathbf{e}_{\theta}} u=|\pa_{\mathbf{e}_{\f{\al_1+\theta}{2}}} u|^2,
\]
we have
\begin{equation}\label{compare pa al1, pa r 1}
\begin{split}
 \left( \f12|\na u|^2+W(u) \right)+\mathbf{e}_{\al_1}\cdot T\mathbf{e}_{\theta}&=\f12(1-\mathbf{e}_{\al_1}\cdot \mathbf{e}_{\theta})|\na u|^2 +\pa_{\mathbf{e}_{\al_1}} u\cdot \pa_{\mathbf{e}_{\theta}} u+(1-\mathbf{e}_{\al_1}\cdot \mathbf{e}_{\theta}) W(u)\\
 &=\frac{|\pa_{\mathbf{e}_{\al_1}}u+\pa_r u|^2}{|\mathbf{e}_{\al_1}+\mathbf{e}_{\theta}|^2}+(1-\mathbf{e}_{\al_1}\cdot \mathbf{e}_{\theta})W(u)\\
 &\geq \frac14|\pa_{\mathbf{e}_{\al_1}}u+\pa_r u|^2\geq0.
\end{split}
\end{equation}
Utilizing \eqref{ene bd S13} and \eqref{id lim F13 2}, for fixed $r_0$ and arbitrarily large $R$,
\begin{equation}
\label{compare pa al1, pa r 2}
\begin{split}
 &\int_{\mathcal S_{13}\cap A_{r_0,R}}\left(\left(\f12|\na u|^2+W(u) \right)+\mathbf{e}_{\al_1}\cdot T\mathbf{e}_{\theta}\right)\,dz\\
 &=J(u,\mathcal S_{13}\cap A_{r_0,R})-\int_{r_0}^R \mathbf{e}_{\al_1} \cdot \mathcal{F}_{13}(r)\,dr\\
 &\leq J(u, \mathcal{S}_{13}\cap B_R^+)-\sigma_{13}(R-r_0)+C\leq C.
\end{split}
\end{equation}
Moreover, \eqref{compare pa al1, pa r 1} implies
\[
 |\pa_{\mathbf{e}_{\al_1}}u|^2
 \leq 2\left(|\pa_{\mathbf{e}_{\al_1}}u+\pa_r u|^2+|\pa_r u|^2\right)
 \leq 8\left(\left(\f12|\na u|^2+W(u) \right)+\mathbf{e}_{\al_1}\cdot T\mathbf{e}_{\theta}\right)+2|\pa_r u|^2.
\]
Combining this with \eqref{compare pa al1, pa r 2} and \eqref{est: dr S13} leads to \eqref{est du al1}, and further \eqref{est: pa x}. 

We continue in the original coordinates. Take $\kappa:= \sqrt{1+k_1^2}+\sqrt{1+k_3^2}$ and apply the divergence theorem to $T$ over the domain 
\[
\mathcal{S}_{13}\cap B_{\kappa r}\cap \{z: z\cdot \mathbf{e}_{\al_1}>r\},\quad r>r_0,
\]
whose boundary consists of four parts:
\begin{align*}
&B_1:=\{r\mathbf{e}_{\al_1}+s\nu_{\al_1}: -k_1r<s<k_3 r\},\quad B_2:=\{s\mathbf{e}_{\varphi_1}: \sqrt{1+k_1^2}r<s<\kappa r\},\\
&B_3:=\{s\mathbf{e}_{\varphi_3}: \sqrt{1+k_3^2}r<s<\kappa r\},\quad B_4:=\{\kappa r\mathbf{e}_{\theta}:\varphi_1<\theta<\varphi_3\},
\end{align*}
to get
\begin{equation}
    \int_{B_1}-T\mathbf{e}_{\al_1}\,d\ch+ \int_{B_2} -T\nu_{\varphi_1}\,d\ch+\int_{B_3} T\nu_{\varphi_3}\,d\ch +\int_{B_4} T\nu\,d\ch=0. 
\end{equation}
By \eqref{exp decay theta i}, the two  radial sides $B_2$ and $B_3$ only contribute $O(e^{-cr})$, so \eqref{id lim F13 2} implies
\[
\int_{B_1}-T\mathbf{e}_{\al_1}\,d\ch=  \sigma_{13}\mathbf{e}_{\al_1}+O(e^{-cr}).
\]
Its tangential and normal components are
\begin{equation}
 \begin{split}
 &\mathbf{e}_{\al_1}\cdot\int_{B_1}-T\mathbf{e}_{\al_1}\,d\ch= \int_{-k_1r}^{k_3r} \left(\f12|\pa_y v|^2-\f12|\pa_x v|^2+W(v)\right)\,dy =\sigma_{13}+O(e^{-cr}),\\
 &\qquad\qquad\nu_{\al_1}\cdot\int_{B_1}-T\mathbf{e}_{\al_1}\,d\ch= \int_{-k_1r}^{k_3r} -\left(\pa_y v\cdot\pa_x v\right)\,dy =O(e^{-cr}).
 \end{split}
\end{equation}
This together with \eqref{decay on paS} implies \eqref{est: G H}. Finally, \eqref{est: ene excess} follows from \eqref{est: G H} and \eqref{est: pa x} immediately. 

\end{proof}

Let
\begin{equation*}
    \mathcal{U}:=\{U_{13}(\cdot-h): h\in\BR\},\ \ \text{be the set of all translates of $U_{13}$}.
\end{equation*}
\begin{equation*}
    \mathcal{A}:=\{w\in H_{loc}^1(\BR,\BR^2): w(-\infty)=a_1,\,w(+\infty)=a_3,\, \int_{\BR}(\f12|w'|^2+W(w))\,dx<\infty\}.
\end{equation*}

For simplicity, we shall henceforth write $U_{13}$ as $U$. We further define 
\begin{equation*}
d_2(w, \mathcal{U}):= \inf\limits_{\tau\in \mathcal{U}}\,\| w-\tau\|_{L^2(\BR;\BR^2)}.
\end{equation*}

We invoke the following result, most of which comes from Schatzman's work \cite{schatzman2002asymmetric},  to show that for most $x>0$, $v|_{\S_x}$, which belongs to $\mathcal{A}$, can be approximated by a translation of $U$.

\begin{lemma}
\label{lem: appr by U}
Consider $w\in\mathcal{A}$ such that $\|w\|_{L^\infty}+\|w'\|_{L^\infty}\leq M$. For any $\e>0$, there exists $\eta_\e>0$ such that
\begin{equation}\label{close ene control L2}
 \int_\BR \left(\f12|w'|^2+W(w)\right)\,dx-\sigma_{13}\leq\eta_\e\quad\Rightarrow\quad d_2(w,\mathcal{U})\leq\e.
\end{equation}
Moreover, there exists $\e_{0}=\e_0(W,M)$ such that if $d_2(w,\mathcal{U})<\e_0$,  then there exists a unique $h=h(w)$ such that 
    \begin{equation*}
        d_2(w,\mathcal{U})=\|w-U(\cdot- h)\|_{L^2},
    \end{equation*}
and
\begin{equation}\label{ene coercivity}
 \int_\BR\left(\f12|w'|^2+W(w)\right)\,dx-\sigma_{13}\ge c\|w-U(\cdot-h)\|_{H^1(\BR)}^2,
\end{equation}
for some $c=c(W)>0$. Moreover, $h(w)$ is a function of class $C^{1}$ in the affine $L^2$ space. 
\end{lemma}

\begin{proof}
The existence, uniqueness, and differentiability of $h(w)$ follow directly from \cite[Lemma 2.1]{schatzman2002asymmetric}. \eqref{close ene control L2} follows from an application of \cite[Theorem 3.1 and Corollary 3.2]{schatzman2002asymmetric}. Note that the original results in \cite{schatzman2002asymmetric} are stated in a two-well setting. Under ($\mathrm{H}_3$) or ($\mathrm{H}_3'$), the strict triangle inequality exclude splitting by another phase. Hence we can adapt the compactness argument to the $a_1$-$a_3$ connection. Hypothesis $(\mathrm H_4)$ then identifies the set of minimizers with $\mathcal U$.

For the coercivity \eqref{ene coercivity}, we follow closely the proof of \cite[Lemma~4.5]{schatzman2002asymmetric}. Indeed, denoting by $q=w-U(\cdot-h)$, one can expand the 1D energy 
\begin{align*}
\int_\BR\left(\f12|w'|^2+W(w)\right)\,dx-\sigma_{13}
 =\frac12\int_\BR
 \left(|q'|^2+q\cdot D^2W(U(y-h))q\right)\,dy
 +\mathcal{R}\geq c\|q\|_{H^1}^2+\mathcal{R},  
\end{align*}
where we used ($\mathrm H_4$) in the last inequality. The remainder term $\mathcal{R}$ has the explicit form
\begin{equation*}
   \mathcal{R}:= \int_\BR\int_0^1 (1-s)\,q(y)\cdot \left[D^2W(U(y-h)+sq(y))-D^2W(U(y-h))\right]q(y)\,ds\,dy.
\end{equation*}
The elementary estimate
\begin{equation}\label{L2 control Linfty}
|q(x)|^3\leq \int_{-\infty}^x 3|q(t)|^2|q(t)|'\,dt\leq 3\|q'\|_{L^\infty}\|q\|^2_{L^2}\leq C(M,W)\|q\|^2_{L^2}
\end{equation}
implies that 
\[
|\mathcal{R}|\leq C(M,W) \|q\|_{L^2}^{\f{8}{3}}.
\]
Thus the remainder term can be absorbed by the coercive term $c\|q\|_{H^1}^2$.

\end{proof}

Define the ``good" set
\begin{equation*}
    \mathcal{G}:= \{x>0: d_2(v(x,\cdot),\mathcal{U})< \e_0\}.
\end{equation*}
From  Lemma \ref{lem: appr by U}, there exists a function $h\in C^1(\mathcal{G};\BR)$ such that 
\begin{equation*}
d_2(v(x,\cdot),\mathcal{U})=\|v(x,\cdot)-U(\cdot-h(x))\|_{L^2}, \quad x\in \mathcal{G}.
\end{equation*}
We have the following identities:
    \begin{equation}
        \label{id: orthogonality}
        \int_{\S_x} (v-U(y-h(x)))\cdot U'(y-h(x)) \,dy=0,\quad  x\in \mathcal{G},
    \end{equation}
    \begin{equation}
        \label{id: formula h'(x)}
        h'(x)=-\f{\int_{\S_x} \pa_xv\cdot U'(y-h(x))\,dy}{\int_{-\infty}^\infty\big[ |U'(y-h(x))|^2 - U''(y-h(x))\cdot(v-U(y-h(x)))  \big]  \,dy}.
    \end{equation}

Identity \eqref{id: orthogonality} follows from the fact that $U(\cdot-h(x))$ minimizes the $L^2$ distance to $v(x,\cdot)$. And \eqref{id: formula h'(x)} follows from differentiating \eqref{id: orthogonality} with respect to $x$.

Using \eqref{est: G H}, \eqref{est: pa x}, \eqref{est: ene excess}, \eqref{close ene control L2} and \eqref{ene coercivity}, we have for $x\in\mathcal{G}$,
\begin{equation}\label{est: h' numerator 1}
    \begin{split}
        &\left|\int_{\S_x} \pa_xv\cdot U'(y-h(x))\,dy\right|\\
        \leq  & \left|\int_{\S_x}\pa_xv\cdot (U'(y-h(x))-\pa_yv)\,dy\right|+\left| \int_{\S_x} \pa_x v\cdot\pa_y v\,dy \right|\\
        \leq & \|\pa_x v(x,\cdot)\|_{L^2}\cdot\|v(x,\cdot)-U(\cdot-h(x))\|_{H^1} +Ce^{-cx}\\
        \leq & \f12 \|\pa_x v(x,\cdot)\|_{L^2}^2+C(\int_{\S_x}\f12|\pa_y v|^2+W(v)\,dy-\sigma_{13}) + Ce^{-cx}\in L^1(\mathcal{G}).
    \end{split}
\end{equation}

\begin{proposition}
\label{prop:translate}
There exists $h_{13}\in\BR$ such that
\begin{equation}\label{L2 convergence}
\lim_{x\to+\infty} \| v(x,\cdot)-U_{13}(\cdot-h_{13})\|_{L^2(\BR)}=0.
\end{equation}
For every $T>0$ and $\beta\in(0,1)$, the above convergence can be upgraded to
\begin{equation}\label{loc C2b convergence}
\lim_{x\to\infty} \|v(x,\cdot)-U_{13}(\cdot-h_{13})\|_{C^{2,\beta}([-T,T])}=0.
\end{equation}
The same conclusion holds for $l_{\alpha_2}$ with $U_{32}$ and a finite shift $h_{32}\in\BR$.
\end{proposition}

\begin{proof}
Let $d_2(x):=d_2(v(x,\cdot),\mathcal{U})$. For $x_2>x_1>x_0$,
\begin{align*}
 &|d_2(x_2)-d_2(x_1)|\leq\|v(x_2,\cdot)- v(x_1,\cdot)\|_{L^2}\\
 &\quad\leq\int_{x_1}^{x_2} (\int_{\S_x} |\pa_x v|^2\,dy)^{\f12}\,dx
 \leq \sqrt{x_2-x_1}\left(\int_{x_0}^\infty  \int_{\S_x} |\pa_x v|^2\,dy \,dx\right)^{1/2}.   
\end{align*}

Thus $d_2$ is uniformly continuous. If $d_2(x) \not\to 0$, there exist $\e,\delta>0$ and disjoint intervals $\{I_n\}_{n=1}^\infty$ such that 
\begin{equation}
   |I_n|\geq \delta, \qquad d_2(x)>\e \ \text{ on }I_n. 
\end{equation}
Using \eqref{close ene control L2} we infer that 
\[
 \int_{x_0}^\infty \left[\int_{\S_x} \left(\f12|\pa_y v|^2+W(v)\right)\,dy-\sigma_{13}\right] \,dx
 \geq\sum_n \delta \eta_\e =\infty,
\]
which yields a contradiction with \eqref{est: ene excess}. Hence $d_2(x)\to 0$.

After enlarging $x_0$ if necessary, we can assume $(x_0,\infty)\subset \mathcal{G}$. The denominator in \eqref{id: formula h'(x)} is bounded below by \[
\|U'\|_{L^2}^2-d_2(x)\|U''\|_{L^2}\geq \f12\|U'\|_{L^2}^2,\quad \forall\,x>x_0.
\]
This, combined with the estimate \eqref{est: h' numerator 1} for the numerator, gives $|h'(x)|\in L^1(x_0,\infty)$. It follows that there exists $h_{13}\in\BR$ such that $h(x)\to h_{13}$ and
\[
 \|v(x,\cdot)-U(\cdot-h_{13})\|_{L^2}
 \leq d_2(x)+\|U(\cdot-h(x))-U(\cdot-h_{13})\|_{L^2}\ \rightarrow\ 0,
\]
which proves \eqref{L2 convergence}. The uniform bound on $|\na v|$ and estimate \eqref{L2 control Linfty} imply 
\begin{equation}\label{inf conv to hete}
     \|v(x,\cdot)-U(\cdot-h_{13})\|_{L^\infty}\to 0.
\end{equation}
Finally, interior elliptic estimates can upgrade the uniform convergence to $C^{2,\beta}_{loc}$ convergence for any $\beta\in(0,1)$, which gives \eqref{loc C2b convergence}.

\end{proof}

\begin{proof}[Proof of Theorem \ref{main thm: refined behavior near interface}]
Proposition \ref{prop:translate} proves the asserted heteroclinic convergence \eqref{main thm 2:convergence to Uij} along interior sharp interfaces. It remains to establish the 
localization of the diffuse interface \eqref{main thm 2:localization of diff interface}.

Fix a sufficiently small $\gamma>0$ and set the diffuse interface as in \eqref{def:diff intef}. For both interfaces $l_{\al_1}$ and $l_{\al_2}$, \eqref{inf conv to hete} implies there exists  $C_\gamma$ such that
\begin{equation}\label{localization in calS}
 D_\gamma\cap(\mathcal{S}_{13}\cup \mathcal{S}_{32})
 \subset\{z:\dist(z,(l_{\al_1}\cup l_{\al_2}))\leq C_\gamma\}.
\end{equation}

To prove \eqref{main thm 2:localization of diff interface}, it remains to show that there exists a sufficiently large $r_0$ such that $D_\g$ does not intersect with $\{r\mathbf{e}_\theta: r>r_0,\ \theta\in(0,\varphi_1)\cup(\varphi_2,\pi)\}$. We focus on 
\[
\mathcal{S}_1:= \{r\mathbf{e}_\theta: r>0, \ \theta\in(0,\varphi_1)\}.
\]
By a similar argument as in the proof of Lemma \ref{lem:sharp ene bd}, there exists $r_0$ such that for any $r>r_0$, we can find a good radius $R\in (r,2r)$ such that
\[
|u(R\mathbf{e}_{\theta})-a_1|\leq \f{\g}{2}, \quad \theta \in(0,\varphi_1).
\]
The idea is that on the semicircle $\pa B_{R}\cap \BR_+^2$, if the total energy (only counting tangential derivative and potential) is sufficiently close to $\sigma_{13}+\sigma_{32}$, then most of the energy should be concentrated near $l_{\al_1}\cup l_{\al_2}$, preventing $u$ from escaping a small neighborhood of energy wells away from the sharp interface as it will cost extra energy. 

We may further assume, by enlarging $r_0$ if necessary, that for every $r>r_0$,
\[
|u(r\mathbf{e}_{\varphi_1})-a_1|\leq \f{\g}{2},\quad |u_0(r)-a_1|\leq \f{\g}{2}.
\]
Thus for any two of these good radii $R_1<R_2$, the annular region $\mathcal{S}_{R_1,R_2}:=\{r\mathbf{e}_\theta: r\in(R_1,R_2),\ \theta\in(0,\varphi_1)\}$ satisfies $|u-a_1|\leq \f{\g}{2}$ on $\pa \mathcal{S}_{R_1,R_2}$, and Lemma \ref{lemma: maximum principle} therefore implies
\[
|u-a_1|\leq \f{\g}{2},\quad \text{in }\mathcal{S}_{R_1,R_2}.
\]
Taking $R_2\to\infty$, we obtain
\[
(\mathcal{S}_1 \setminus B_{r})\cap D_\g=\emptyset, \quad \text{ for some }r>0.
\]
Similarly, we have 
\[
(\{r\mathbf{e}_\theta: \theta\in(\varphi_2,\pi)\} \setminus B_{r})\cap D_\g=\emptyset, \quad \text{ for some }r>0.
\]
Hence we have proved Theorem \ref{main thm: refined behavior near interface} for the case of two interior interfaces. The same argument applies to the single interior interface case. 

\end{proof}

\bibliographystyle{acm}
\bibliography{half-space}
\end{document}